\documentclass[11pt,leqno]{article}
\usepackage{amsmath,amsthm}
\usepackage{amsfonts}
\usepackage{mathrsfs}
\usepackage{hyperref}
\usepackage{amssymb}
\usepackage{epsfig}
\usepackage{float}
\usepackage{lineno}
\usepackage{subfig}
\usepackage{graphicx}
\usepackage{bm}
\usepackage{booktabs}
\usepackage{color}
\numberwithin{equation}{section} 
\newtheorem{theorem}{\bf Theorem}[section]
\newtheorem{example}{\bf Example}[section]

\newtheorem{remark}{\bf Remark}[section]
\newtheorem{lemma}{\bf Lemma}[section]

\newcommand{\norm}[1]{\left\lVert #1\right\rVert}

\newsavebox{\savepar}		 

\begin{document}

	\title {\bf \large  Penalty-Based Feedback Control and Finite Element Analysis for the Stabilization of Nonlinear Reaction-Diffusion Equations }
	%	 \author{  Sudeep Kundu and Shishu Pal Singh \\
		%		  }
	\author{
		Sudeep Kundu\thanks{Department of Mathematical Sciences, Rajiv Gandhi Institute of Petroleum Technology. Email: sudeep.kundu@rgipt.ac.in}  
		\;and
		Shishu Pal Singh\thanks{Department of Mathematical Sciences, Rajiv Gandhi Institute of Petroleum Technology. Email: shishups@rgipt.ac.in}
	}
	%%
	%\linenumbers
	\maketitle
	\begin{abstract}	
		In this work, first we employ a penalization technique to analyze a  Dirichlet boundary feedback control problem pertaining to reaction-diffusion equation. We establish the stabilization result of the equivalent Robin problem in the \(H^{2}\)-norm with respect to the penalty parameter. Furthermore, we prove that the solution of the penalized control problem converges to the corresponding solution of the Dirichlet boundary feedback control problem as the penalty parameter \(\epsilon\) approaches zero. A \(C^{0}\)-conforming finite element method is applied to this problem for the spatial variable while keeping the time variable continuous. We discuss the stabilization of the semi-discrete scheme for the penalized control problem and present an error analysis of its solution. Finally, we validate our theoretical findings through numerical experiments including showing that  penalized solution converges to the original solution.
		
	\end{abstract} 
	
	%93C20.
	\section{Introduction.}	
	We consider the nonlinear reaction-diffusion equation, known as the Chafee-Infante (CI) equation, with the Dirichlet boundary control \cite{19M1252235}, given by
	\begin{align}
		\label{ch1}
		y_{t}-\nu y_{xx} &= \alpha y - \delta y^{3}, \quad (x, t) \in (0, 1) \times (0, \infty), \\
		y(t,0) &= 0, \quad t \in (0, \infty), \label{ch12}\\
		y(t,1) &= u(t), \quad t \in (0, \infty), \\
		y(0,x) &= y_{0}(x), \quad x \in (0, 1),
		\label{ch2}
	\end{align}
	where the reaction coefficients \( \alpha \) and \( \delta \), as well as the diffusion coefficient \( \nu \), are positive constants, and the control input \( u(t) \) is defined as  
	\begin{align}
		u(t) = -r \int_{0}^{1} x y(t,x) \,dx, \label{c}
	\end{align}
	where \( r \) is a positive constant and \( y_{0}(x) \) is a given function. In \eqref{ch1}, subscripts denote differentiation i.e. \(y_{t}(t,x)=\frac{\partial y(t,x)}{\partial t}\) and \(y_{xx}(t,x)=\frac{\partial^{2} y(t,x)}{\partial x^{2}}\).

    The reaction-diffusion equation \eqref{ch1} is used in the study of many applications for example in material science is the phase separation process, where \(y=y(t,x)\) is defined as the concentration of one of the possible phase \cite{MR3590173} and also see \cite{ CHAFEE1975111, SADAF2023107097, MR953967} for application. The controllers can be highly relevant within a reverse engineering approaches for such reaction-diffusion equation \cite{MR3590173}.
      The stabilization of reaction-diffusion systems has attracted significant attention in recent decades; see, for example, \cite{9304185, MR4077349, MR4020518, MR4721995, MR3762670} and references therein. The CI equation is sometimes referred to as the Allen-Cahn equation, the real Ginzburg-Landau equation, or the Newell-Whitehead equation \cite{Martin2023}. The stabilization of this type of equation with Dirichlet boundary control has been well studied; see, for example, \cite{19M1252235, 18M1213129, LHACHEMI2022109955}. Additionally, stabilization with state delay under Dirichlet boundary control, with Robin boundary conditions, and with delay in the Dirichlet boundary control can be found in \cite{7426756}, \cite{9099005}, and \cite{8392394}, respectively. Furthermore, the stabilization of this equation with distributed feedback control and Neumann boundary feedback control has been discussed in \cite{azouani2013feedback} and \cite{yang2019asymptotic}, respectively.
     
One of the major challenges for second-order partial differential equations with Dirichlet boundary conditions is that the problem is non-variational and one cannot identify a suitable functional framework without using an appropriate boundary lifting. The global exponential stability of the closed loop system \eqref{ch1}–\eqref{c} in the $H^{1}$-norm has been shown in \cite{19M1252235}, by constructing a control Lyapunov functional. Moreover, they have discussed the existence and uniqueness of the solution to the closed loop system \eqref{ch1}–\eqref{c}.     
     Previous studies have employed a boundary penalization to the uncontrolled elliptic equation in  \cite{MR351118, MR853660, gudi2014}, where the boundary penalization techniques converts a Dirichlet boundary condition into a Robin boundary. In the work \cite{MR853660}, Barrett et al. have discussed the error analysis with respect to the penalty parameter, using a finite element method. Using Babu\v ska’s penalty method, the optimal error estimates are established in \cite{gudi2014}.
	Moreover, the boundary penalization technique has been widely used in a finite element analysis across various applications, such as flow problems \cite{MR660571}, unsteady and steady Navier-Stokes equations \cite{CAREY1984183, MR2136999, HE2010708, MR2912615, 0732016}, viscoelastic fluid models \cite{MR4399840, MR2754643, MR2818046}, and magnetohydrodynamics equations \cite{JDeng2019, MR3486328, MR3969002}, among others.
	
	In the literature, the approximation of Dirichlet boundary control problems using boundary penalization technique has been well established in the context of optimal control problems. In \cite{S0363012996304870, S1064827597325153}, the authors have solved the steady-state Navier-Stokes equation using the boundary penalization technique. Further, the error analysis of penalized problem is discussed with the help of  a finite element approximation in \cite{S1064827597325153}. Using a mixed Galerkin finite element method, Ravindran \cite{ravindran2017finite} studied boundary penalization control for the unsteady Navier-Stokes equation and established optimal error estimates in the \( L^{\infty}(L^{2}) \)-norm for the penalized system. In \cite{18M1233716}, Ravindran extended this work  for the nonstationary magnetohydrodynamics equation using a boundary penalization technique. Additionally, the convergence analysis of the penalty function was examined, showing that the penalized control problem converges to the original control problem as \( \epsilon \to 0 \), where \( \epsilon \) is the penalty parameter.  
	In the work \cite{belgacem2003singular}, error estimates were derived in terms of \( \epsilon \), showing that the order of convergence is \( \mathcal{O}(\epsilon^{\frac{1}{2}}) \) for regular domains and \( \mathcal{O}(\epsilon^{\frac{1}{2}-\delta}) \) for all \( \delta > 0 \) in convex polygons. After that Eduardo et al. \cite{casas2009penalization} establish an error estimate of order \( \mathcal{O}(\epsilon^{1-\frac{1}{p}})\) for some \(p>2\) with control constraints. Additionally,  a Robin boundary penalization for other second order parabolic equations has been shown in \cite{MR1874072, ben2003penalized, ben2011dirichlet}, where convergence analysis of the penalized solution is discussed. Moreover, the authors of \cite{ben2003penalized} discuss the convergence rate with respect to \(\epsilon\).
	
	In this article, we examine the boundary penalization techniques for the closed loop system \eqref{ch1}–\eqref{c}. Furthermore, we demonstrate the result of \cite{19M1252235} numerically.
    Since this type of boundary control does not naturally appear in the weak formulations, we use a boundary penalization to transform the above closed loop system \eqref{ch1}–\eqref{c} into an equivalent Robin boundary problem \eqref{p1}–\eqref{p2}, which is defined later. This transformation is advantageous, as Robin boundary conditions are more amenable to finite element implementations than Dirichlet boundary conditions, as is also the case for boundary feedback control problems. To the best of our knowledge, this work presents the first application of  penalization techniques to Dirichlet boundary feedback control problems related to the reaction-diffusion type equation.  
	
	In this work, we present the following contributions:  
	\begin{itemize}
		
		\item First, we derive some {\it{a priori}} bounds  for the closed loop system \eqref{ch1}-\eqref{c} following \cite{19M1252235}. Then we establish the exponential stabilization of the penalized feedback control problem and provide regularity results that are crucial for deriving error estimates for both the state variable and control input. 
		\item We conduct a convergence analysis and show that the solution of the penalized feedback control problem approaches that of the corresponding Dirichlet boundary control problem as the penalty parameter \( \epsilon \) tends to zero.  
		\item Using a \(C^{0}\)-conforming finite element method, the global stabilization of the semi-discrete scheme is discussed. Further, we derive error estimates for the semi-discrete scheme of the state variable and control input in the penalized control problem.  
		\item We provide some numerical experiments to show the behavior of state and control trajectory  as well as their order of convergence. Moreover, the convergence behavior of the penalized control problem is examined, including the numerical solution of the closed loop system \eqref{ch1}-\eqref{c}.  
		
	\end{itemize}
	In this paper, we adopt the following notations from \cite{MR1625845}:
	\begin{itemize}
		\item \(C^{k}([0,T]; (0,1))\), where \(k\geq 1,\) denotes the class of continuous function on \([0, T]\), which have continuous derivatives of order \(k\) on \([0,T]\) and also take values in \((0,1)\).
		\item  The space $ L^{p}((0,T);X)$ consist of all strongly measurable functions $ v:[0,T]\rightarrow X $, equipped with the norm:
		\begin{align*}
			\norm{v}_{L^{p}((0,T);X)}:=
			\begin{cases}
				\left(\int_{0}^{T}\norm{v}_{X}^{p}dt\right)^{\frac{1}{p}}< \infty, &\text{if} \ 1 \leq p<\infty, \\
				\mathop{\mathrm{ess\, sup}} \limits_{0\leq t\leq T} (\norm{v(t)}_{X})< \infty, \medspace &\text{if}\  p=\infty,
			\end{cases}
		\end{align*}
		where $ X $ denotes a Banach space with norm $ {||.||}_{X}$. We use the notation $ {( .,. )}$, which represents the $ L^{2}$- inner product with the corresponding norm denoted by \( \norm{\cdot} \). We denote to \(L^{p}((0,T);X) \) by \(L^{p}(X) \).
		\item The Sobolev space is defined as 
		\[
		H^{m}(0,1) = \left\{ v \mid v\in L^{2}(0,1), \ \frac{\partial^{i} v}{\partial x^{i}} \in L^{2}(0,1), \quad \text{for } i = 1, 2, \ldots, m \right\}.
		\]  
		Additionally, we define  
		\( H_{\{0\}}^{1}(0,1) := \{ v \in H^{1}(0,1) \mid v(0) = 0 \}.\)

		%    \item \textbf{Gagliardo-Nirenberg-Sobolev Inequality:} If \( p \in (2, \infty) \), then  
		%    \[
		%        \norm{v}_{L^{p}} \leq C \norm{v_{x}}^{l} \norm{v}^{1-l},
		%    \]
		%    where \( l = \frac{p-2}{2p} \), and \( C \) is a positive constant.
		
		\item \textbf{Young's Inequality:} For all \( a, b > 0 \) and \( c_{0} > 0 \), we have  
		\[
		ab \leq \frac{c_{0}}{2} a^{2} + \frac{1}{2c_{0}} b^{2}.
		\]
	\end{itemize}
	
	\begin{itemize}
		\item \textbf{Sobolev Embedding:}[{Chapter 5}] The following embedding holds:  
		\[
		H^{1}(0,1) \hookrightarrow L^{p}(0,1), \quad 2\leq p\leq \infty.
		\]
		
	\end{itemize}
	Throughout this paper, \( C \) denotes a generic positive constant.\\
	The rest of the paper is organized as follows: Section~$2$ presents some regularity results and some exponential bounds. In Section~$3$, we discuss the formulation and stabilization of the penalized control problem. Section~$4$ analyzes the uniqueness and convergence analysis of the penalized control problem. Section~$5$ focuses on the finite element method for the semi-discrete solution, providing error estimates for both the state variable and control input. Finally, in Section~$6$, we verify our theoretical conclusions using numerical experiments.
	\section{Dirichlet Control Problem.}
	Here, we discuss some regularity results for the closed loop system \eqref{ch1}-\eqref{c}. The closed loop system \eqref{ch1}-\eqref{c} has been studied in a more general framework in \cite{19M1252235}.
	
	We introduce the variational formulation of the closed loop system \eqref{ch1}-\eqref{c} is to find \( y \in H_{\{0\}}^{1} \) such that
	\begin{align}
		\label{1.6}
		(y_{t}, \chi) + \nu (y_{x}, \chi_{x}) = \nu y_{x}(1)\chi(1) + \alpha(y, \chi) - \delta(y^{3}, \chi), \quad \text{for all } \chi \in H_{\{0\}}^{1},
	\end{align}
	with the boundary condition
	\[
	y(t,1) = u(t), \quad \text{where } u(t) = -r\int_{0}^{1} x y(t,x)\, dx=-r(x, y),
	\]
	where \(y=y(t,x)\). We write $y(t,x)$ as \(y\) when no confusion arises.
	
	Now, we state the regularity results for the closed loop system \eqref{ch1}-\eqref{c}.
    
	\begin{theorem}
		\label{thr.1}
		Let \(y_{0}\in H^{2}(0,1)\) with \(y_{0}(0)=0, \  y_{0}(1)=-r(x,y_{0})\). Then there exists   \(t_{m}\in(0, \infty]\) and a unique mapping \(y\in C^{0}([0, t_{m}), [0,1])\cap C^{1}((0, t_{m}); L^{2}(0,1))\) with \(y(0,x)=y_{0}, \ y\in H^{2}(0,1)\) for all \(t\in[0, t_{m})\),  for which the mapping \(t\rightarrow \norm{y_{x}+rx y}^{2}\) is \(C^{1}\) on \([0, t_{m})\), and for which the closed loop system \eqref{ch1} holds for all \(t\in(0, t_{m}) \) and \eqref{ch12}-\eqref{c} holds for all \(t\in[0, t_{m})\).
        \end{theorem}
		\begin{proof}
			For the proof, see \cite{19M1252235}.
		\end{proof}
	
	The following lemma holds under the regularity results stated in Theorem \ref{thr.1}. These results will assist in the convergence analysis between the penalized control problem \( y^{\epsilon} \) and the corresponding Dirichlet boundary control problem \( y \) for the closed loop system \eqref{ch1}-\eqref{c}.
	
	\begin{lemma}
		\label{L1.1}
		Suppose that the assumption of Theorem \ref{thr.1} true. Then, there exists a decay rate \( 0 < \gamma \leq \dfrac{3\left(2\nu - \alpha - \frac{1}{3}(r\alpha + r^2\nu)\right)}{r+3} \), and the condition
		\[
		\dfrac{\alpha}{\nu} \leq \dfrac{6 - r^2}{r + 3}, \quad r <\min \left\{ \frac{1}{\norm{x}_{L^{4}} \norm{x}_{L^{\frac{4}{3}}}},\, \sqrt{6} \right\},
		\]
		such that the following holds
		\begin{align*}
		\norm{y}^{2} + r(x,y(t))^{2}+\nu \norm{y_{x}}^{2}
		+ \frac{\delta}{2} \norm{y}_{L^{4}}^{4} &+ \beta^{*} e^{-2\gamma t} \int_{0}^{t} e^{2\gamma s} \left( \norm{y_{x}(s)}^{2} + \norm{y(s)}_{L^{4}}^{4} \right)\, ds 
		\\&+ e^{-2\gamma t} \int_{0}^{t} e^{2\gamma s} \left( \norm{y_{t}(s)}^{2} + r(x,y_{t}(s) )^{2} \right) ds\leq C e^{-2\gamma t} \norm{y_{0}}_{1}^{2},
		\end{align*}
		where
		\[
		0 < \beta^{*} = \min \left\{ 2\nu - \alpha - \frac{\gamma}{3}(r+3) - \frac{1}{3}(r\alpha + r^2\nu), \, \delta\left(1 - r\norm{x}_{L^{4}} \norm{x}_{L^{\frac{4}{3}}} \right) \right\},
		\] $C$ is a positive constant, and \[\norm{x}_{L^{4}} \norm{x}_{L^{\frac{4}{3}}}
 = \left( \frac{1}{5} \right)^{\frac{1}{4}} \left( \frac{3}{7} \right)^{\frac{3}{4}} \approx 0.3542.\]
	\end{lemma}
	\begin{proof}
	For details of the proof see \cite{19M1252235}.
	\end{proof}
	\begin{remark} 
		\label{r2.2}
		From Lemma \ref{L1.1}, we conclude that the system is stable around an equilibrium solution \(y^{\infty}=0\), when \(r=0\). There exists a decay rate \( 0 < \gamma \leq 2\nu - \alpha \) with the condition
		$\frac{\alpha}{\nu} \leq 2, $
		such that  the following estimate holds
		\[
		\norm{y}^{2} + \min \{(2\nu - \alpha - \gamma ), \, \delta \} e^{-2\gamma t} \int_{0}^{t} e^{2\gamma s} \left( \norm{y_{x}(s)}^{2} + \norm{y(s)}_{L^{4}}^{4} \right)\, ds \leq C e^{-2\gamma t} \norm{y_{0}}^{2}.
		\]
	\end{remark}
	Therefore, the system \eqref{ch1} is stabilizable with zero Dirichlet boundary condition under the assumption $\frac{\alpha}{\nu} \leq 2 $.
	\begin{remark}
		\label{r1.2}
		Applying the Cauchy-Schwarz inequality, from \eqref{ch1} we have
		\[\nu \norm{y_{xx}}^{2}\leq C(\norm{y_{t}}^{2}+ \norm{y}^{2}+ C\norm{y}_{1}^{6}).\]
		Hence, using Lemma~\ref{L1.1} and the Sobolev embedding, we obtain
		\[\int_{0}^{t}\nu e^{2\gamma s}\norm{y_{xx}(s)}^{2}ds \leq  C(\norm{y_{0}}_{1}^{2}) e^{-2\gamma t}.\]
	\end{remark}

     The following regularity results are needed for the convergence analysis of the penalized controlled system.
    \begin{theorem}
    \label{th2.2}
        Let \(y_{0}\in H^{3}(0,1)\cap H_{\{0\}}^{1}(0,1)\) and the regularity result stated in Theorem \ref{thr.1} hold. Then, there exists a positive constant  \(C=C(\norm{y_{0}}_{1})\) such that
        \[\norm{y_{t}}^{2}+r(x,y_{t}(t))^{2}+e^{-2\gamma t}\int_{0}^{t}e^{2\gamma s}(\nu \norm{y_{xt}(s)}^{2}+ 3\delta\norm{y(s)y_{t}(s)}^{2})ds\leq Ce^{-2\gamma t}\norm{y_{0}}_{2}^{2},\]
        and 
        \[\nu \norm{y_{xt}}^{2}+e^{-2\gamma t}\int_{0}^{t}e^{2\gamma s}\Big(\norm{y_{tt}(s)}^{2}+r(x,y_{tt}(s))^{2}\Big)ds\leq Ce^{-2\gamma t}\norm{y_{0}}_{3}^{2}.\]
    \end{theorem}
	\begin{proof}
	    For the proof see Appendix.
	\end{proof}
	
    As a consequence of Remark \ref{r1.2} and Theorem \ref{th2.2}, we get the estimate \(\norm{y_{xx}}\), which decays exponentially. 
	\begin{remark}
		\label{r1.3}
		From \eqref{1.7}, we have
		\[\nu y_{xxt}=y_{tt}-\alpha y_{t}+3\delta y^{2}y_{t}.\]
		Therefore, we obtain
		\[\nu \norm{y_{xxt}}^{2}\leq C(\norm{y_{tt}}^{2}+ \norm{y_{t}}^{2}+ \norm{y_{t}}^{2}\norm{y}_{\infty}^{4}).\]
		Using Theorem \ref{th2.2}, we arrive at 
		\[\int_{0}^{t}\nu e^{2\gamma s}\norm{y_{xxt}(s)}^{2}ds \leq  Ce^{-2\gamma t}\norm{y_{0}}_{3}^{2}.\]
	\end{remark}
	
	\section{Penalized Control Problem.}
	Next, we discuss the formulation of the Dirichlet boundary feedback control problem using a boundary penalization procedure.
	
	Using a boundary penalization technique, the Dirichlet boundary feedback control problem given by \eqref{ch1}–\eqref{c} can be reformulated as follows
	\[
	\frac{\partial y^{\epsilon}}{\partial x}(t,1)+ \frac{1}{\epsilon}y^{\epsilon}(t,1)=\frac{1}{\epsilon}u^{\epsilon}(t),
	\]
	where $   \epsilon>0 $ is a penalty parameter.
	
	Now, for each \( \epsilon > 0 \), we introduce the following Robin boundary control problem: find \( y^{\epsilon} = y^{\epsilon}(t,x) \), with \( x \in (0,1) \) and \( t > 0 \), such that
	\begin{align}
		\label{p1}
		&y^{\epsilon}_{t}=\nu y^{\epsilon}_{xx}+ \alpha y^{\epsilon}- \delta (y^{\epsilon})^{3},\\[1mm]&
		y^{\epsilon}(t,0)=0, \ 
		\epsilon \frac{\partial y^{\epsilon}}{\partial x}(t,1)+ y^{\epsilon}(t,1)=u^{\epsilon}(t),\\[1mm]&
		y^{\epsilon}(0,x)=y_{0}(x),
		\label{p2}
	\end{align}
	where the penalized control input is defined as $ u^{\epsilon}(t)=-r\int_{0}^{1}xy^{\epsilon}(t,x) dx$. Therefore, when the penalty parameter \(\epsilon \rightarrow 0 \), the penalized feedback control problem \eqref{p1}-\eqref{p2} becomes the original closed loop system \eqref{ch1}-\eqref{c}.
	
	The variational formulation of the penalized feedback control problem \((\ref{p1})\)–\((\ref{p2})\) is as follows: find \( y^{\epsilon}(t) \in H_{\{0\}}^{1}(0,1) \) such that
	\begin{align}
		\label{1.12}
		&(y_{t}^{\epsilon}, \chi)+\nu (y_{x}^{\epsilon}, \chi_{x}) + \frac{\nu}{\epsilon}y^{\epsilon}(t,1) \chi(1)+ \delta((y^{\epsilon})^{3}, \chi)= \frac{\nu}{\epsilon}u^{\epsilon}(t)\chi(1)+\alpha(y^{\epsilon}, \chi), \ \ \forall \ \chi\in H_{\{0\}}^{1}.
	\end{align}
	Throughout this paper, the following assumptions hold:
	
	\textbf{(A1).} \label{A1} We assume that 
	\[ \frac{\alpha}{\nu}\leq 2 \frac{(3\epsilon - r^{2})}{3\epsilon},\ \text{and} \ \frac{r^{2}}{3\epsilon}\leq \frac{1}{2}   ,\]
	
	with the decay rate \begin{align}
		\label{2.1}
		0< \gamma\leq \nu- \alpha.
	\end{align}
	In the following theorem, the existence and uniqueness of the penalized feedback control problem \eqref{p1}–\eqref{p2} is discussed.
	\begin{theorem}(Regularity).
    \label{th3.1}
		Let $ y_{0} \in H^{3}(0,1)\cap H_{\{0\}}^{1}(0,1) $ and \(y_{0}(1)=-r( x, y_{0})\). Under assumption $(A1)$, there exists a unique weak solution \( y^{\epsilon} \) such that
		\begin{align*}
			y^{\epsilon} \in L^{\infty}&((0,T),H^{2}(0, 1))\cap L^{2}((0, T), H^{2}(0,1))\cap  L^{\infty} \left((0, T), L^{4}(0,1) \right)\cap L^{4} \left((0, T), L^{4}(0,1) \right), \\&  \text{and} \
			\frac{\partial y^{\epsilon}}{\partial t} \in L^{\infty} \big((0, T), H^{1}(0,1)\big)\cap L^{2}\big((0, T), H^{2}(0,1) \big) .
		\end{align*}
	\end{theorem}
	\begin{proof} For the proof of existence see Appendix. The uniqueness of the closed loop system \eqref{p1}–\eqref{p2} is proved in Section $4$.
		\end{proof}

	As a consequence of Theorem \ref{th3.1} the following regularity results hold:
	\begin{align}
\label{r}
\norm{y^{\epsilon}}_{2}^{2}+\norm{y_{t}^{\epsilon}}_{1}^{2}&+\norm{y^{\epsilon}}_{L^{4}}^{4}+(y^{\epsilon}(t, 1))^{2}+(y_{t}^{\epsilon}(t, 1))^{2}\nonumber\\&+\int_{0}^{t}\left(\norm{y^{\epsilon}_{t}(s)}_{2}^{2}+\norm{y^{\epsilon}(s)}_{2}^{2}+\norm{y^{\epsilon}(s)}_{L^{4}}^{4}+(y_{t}^{\epsilon}(s, 1))^{2}+(y^{\epsilon}(s, 1))^{2}\right)ds\leq C(\norm{y_{0}}_{3}),
	\end{align}
	where $ y_{0}\in H^{3}\cap H_{\{0\}}^{1}$ and \(y_{0}(1)=-r( x, y_{0})\).

The results of the rest of the paper hold under assumption $(A1)$ and the regularity result \eqref{r}.
	\subsection{Stabilization.}
	This subsection contains the stabilization of the penalized feedback control problem. Further, some regularity results are established. These regularity results will help in the proof of error analysis to the state variable and control input for the penalized feedback control problem.
	\begin{lemma}
		\label{L2.1}
		Under assumption $(A1)$ and the regularity result \eqref{r}, the following estimate holds
		\begin{align*}
			\norm{y^{\epsilon}}^{2}+ \beta e^{-2\gamma t} \int_{0}^{t}e^{2\gamma s}\left(\norm{y_{x}^{\epsilon}(s)}^{2} + (y^{\epsilon}(s,1))^{2}\right)ds+ 2\delta e^{-2\gamma t}\int_{0}^{t} e^{2\gamma s}\norm{y^{\epsilon}(s)}_{L^{4}}^{4}ds\leq e^{-2\gamma t}\norm{y_{0}}^{2},
		\end{align*}
		where \( \beta= \min\{ (2\nu -\gamma-\frac{2\nu r^{2}}{3\epsilon}-\alpha), \frac{\nu}{\epsilon}\}>0. \)
	\end{lemma}
	\begin{proof}
		Substituting $ \chi= y^{\epsilon} $ in \eqref{1.12}, we get
		\begin{align}
			\label{2.8}
			\frac{1}{2}\frac{d}{dt}\norm{y^{\epsilon}}^{2}+\nu\norm{y_{x}^{\epsilon}}^{2}+ \frac{\nu}{\epsilon}(y^{\epsilon}(t,1))^{2}+ \delta \norm{y^{\epsilon}}_{L^{4}}^{4}=\frac{\nu}{\epsilon}u^{\epsilon}(t)y^{\epsilon}(t,1)+ \alpha \norm{y^{\epsilon}}^{2}.
		\end{align}
		Using Young's inequality, the first term on the right side of \eqref{2.8} is constrained by
		\begin{align}
			\label{2.9}
			\frac{\nu}{\epsilon}u^{\epsilon}(t)y^{\epsilon}(t,1)=\frac{-\nu r}{\epsilon}y^{\epsilon}(t,1)\int_{0}^{1}xy^{\epsilon}(t,x)dx\leq \frac{2\nu r^{2}}{3\epsilon}\norm{y^{\epsilon}}^{2}+ \frac{\nu}{8\epsilon}(y^{\epsilon}(t,1))^{2}.
		\end{align}
		From \eqref{2.9} and \eqref{2.8}, we arrive at
		\begin{align*}
			\frac{1}{2}\frac{d}{dt}\norm{y^{\epsilon}}^{2}+\nu\norm{y_{x}^{\epsilon}}^{2}+ \frac{\nu}{\epsilon}(y^{\epsilon}(t,1))^{2}+ \delta \norm{y^{\epsilon}}_{L^{4}}^{4}\leq (\frac{2\nu  r^{2}}{3\epsilon} + \alpha)\norm{y^{\epsilon}}^{2}+ \frac{\nu}{8\epsilon}(y^{\epsilon}(t,1))^{2}.
		\end{align*}
		Multiplying to the above inequality by $ 2e^{2\gamma t}$ and using the Poincar\'e inequality, it follows that
		\begin{align}
			\label{2.110}
			\frac{d}{dt}(\norm{y^{\epsilon}}e^{\gamma t})^{2}+ \left( 2\nu -\gamma-\frac{2\nu r^{2}}{3\epsilon}-\alpha \right)e^{2\gamma t}\norm{y_{x}^{\epsilon}}^{2}+  \frac{7\nu}{4\epsilon}e^{2\gamma t}(y^{\epsilon}(t,1))^{2}+ 2\delta e^{2\gamma t}\norm{y^{\epsilon}}_{L^{4}}^{4}\leq 0.
		\end{align}
		Now, under the assumption ({A1}) with \( 0<\gamma \leq \nu - \alpha \), the term inside the parentheses remains positive, ensuring exponential decay.
		Integrating \eqref{2.110} with respect to time from $ 0 $ to $ t $, and multiplying the resulting inequality by \( e^{-2\gamma t} \), this completes the proof.
	\end{proof}
	\begin{remark}
		\label{r2.1}
		From the above lemma, we can write
		\begin{align*}
			\norm{y^{\epsilon}}^{2}+ \beta \int_{0}^{t}(\norm{y_{x}^{\epsilon}(s)}^{2}+ (y^{\epsilon}(s,1))^{2})ds + 2\delta \int_{0}^{t} \norm{y^{\epsilon}(s)}_{L^{4}}^{4}ds\leq \norm{y_{0}}^{2}.
		\end{align*}
		Using the Poincar\'e inequality and multiplying by $ e^{2\gamma t}$, it follows that after integration over $ [0, t] $
		\begin{align*}
			\int_{0}^{t}e^{2\gamma s}\norm{y^{\epsilon}(s)}^{2}\leq C \int_{0}^{t}e^{2\gamma s}\norm{y_{x}^{\epsilon}(s)}^{2}\leq C\norm{y_{0}}^{2}.
		\end{align*}
	\end{remark}
	\begin{lemma}
		\label{L2.2}
		Assume that assumption $(A1)$ and the regularity result \eqref{r} hold. Then, there exists a positive constant $ C $ such that
		\begin{align*}
			\nu\norm{y_{x}^{\epsilon}}^{2} + \frac{\nu}{2\epsilon} (y^{\epsilon}(t,1))^{2}+ \frac{\delta}{2}\norm{y^{\epsilon}}_{L^{4}}^{4} + \frac{e^{-2\gamma t}}{3}\int_{0}^{t}e^{2\gamma s}\norm{y_{t}^{\epsilon}(s)}^{2}ds\leq C(\norm{y_{0}}_{1})e^{-2\gamma t}.
		\end{align*}
	\end{lemma}
	\begin{proof}
		Choose $ \chi =y_{t}^{\epsilon} $ in \eqref{1.12}, we obtain 
		\begin{align*}
			\frac{1}{2} \frac{d}{dt}\left(\nu\norm{y_{x}^{\epsilon}}^{2} + \frac{\nu}{\epsilon} (y^{\epsilon}(t,1))^{2}+ \frac{\delta}{2}\norm{y^{\epsilon}}_{L^{4}}^{4}\right) +\norm{y_{t}^{\epsilon}}^{2} = \frac{\nu}{\epsilon}\frac{d}{dt}\Big(u^{\epsilon}(t) y^{\epsilon}(t,1)\Big)- \Big(u_{t}^{\epsilon}(t) y^{\epsilon}(t,1)\Big) +\alpha (y^{\epsilon}, y_{t}^{\epsilon}).
		\end{align*}
		An application of the Young's inequality, we arrive at
		\begin{align*}
			\frac{1}{2} \frac{d}{dt}\left(\nu\norm{y_{x}^{\epsilon}}^{2} + \frac{\nu}{\epsilon} (y^{\epsilon}(t,1))^{2}+ \frac{\delta}{2}\norm{y^{\epsilon}}_{L^{4}}^{4}\right) + \frac{1}{6}\norm{y_{t}^{\epsilon}}^{2}\leq \frac{\nu}{\epsilon}\frac{d}{dt}\Big(u^{\epsilon}(t) y^{\epsilon}(t,1)\Big) + \frac{\alpha^{2}}{2}\norm{y^{\epsilon}}^{2}+\frac{\nu^{2} r^{2}}{4\epsilon^{2}} (y^{\epsilon}(t,1))^{2}.
		\end{align*}
		% Multiplying by $ 2e^{2\gamma t} $, it follows that
		% \begin{align*}
		% 	\frac{d}{dt}\left(\Big(\nu\norm{y_{x}^{\epsilon}}^{2} + \frac{\nu}{\epsilon} (y^{\epsilon}(t,1))^{2}+ \frac{\delta}{2}\norm{y^{\epsilon}}_{L^{4}}^{4}\Big)e^{2\gamma t}\right)&+  \frac{1}{3}e^{2\gamma t}\norm{y_{t}^{\epsilon}}^{2}\\[1mm]&\leq 2\gamma\left(\nu\norm{y_{x}^{\epsilon}}^{2} + \frac{\nu}{\epsilon} (y^{\epsilon}(t,1))^{2}+ \frac{\alpha}{2}\norm{y^{\epsilon}}_{L^{4}}^{4}\right)e^{2\gamma t}\\[1mm]& \ +\frac{\nu}{\epsilon}\frac{d}{dt}\left(e^{2\gamma t}u^{\epsilon}(t) y^{\epsilon}(t,1)\right)- \frac{2\nu \gamma}{\epsilon}\left(e^{2\gamma t}u^{\epsilon}(t) y^{\epsilon}(t,1)\right)\\[1mm]& \ \ +\alpha^{2}e^{2\gamma t}\norm{y^{\epsilon}}^{2}+\frac{\nu^{2} r^{2}}{2\epsilon^{2}}e^{2\gamma t} (y^{\epsilon}(t,1))^{2}.
		% \end{align*}
       Multiplying by $ 2e^{2\gamma t} $ and applying Young's inequality to the resulting inequality, then integrating with respect to time over $ [0, t]$, we arrive at
		\begin{align*}
			\nu\norm{y_{x}^{\epsilon}}^{2} + \frac{\nu}{2\epsilon} (y^{\epsilon}(t,1))^{2}+ \frac{\delta}{2}\norm{y^{\epsilon}}_{L^{4}}^{4} &+ \frac{e^{-2\gamma t}}{3}\int_{0}^{t}e^{2\gamma s}\norm{y_{t}^{\epsilon}(s)}^{2}ds\\[1mm]&\leq Ce^{-2\gamma t}\int_{0}^{t}\left(\norm{y^{\epsilon}(s)}^{2}+\norm{y_{x}^{\epsilon}(s)}^{2}+ (y^{\epsilon}(s,1))^{2} + \norm{y^{\epsilon}(s)}_{L^{4}}^{4}\right)ds\\[1mm]& \ \ +\frac{\nu r^{2}}{6\epsilon}\norm{y^{\epsilon}}^{2}+ Ce^{-2\gamma t}\norm{y_{0}}_{1}^{2}.
		\end{align*}
		An application of Lemma \ref{L2.1} completes the proof.
	\end{proof}
	
	\begin{lemma}
		\label{L2.3}
		Under assumption $(A1)$ and the regularity result \eqref{r}, there exists a positive constant $ C $ such that
		\begin{align*}
			\norm{y_{t}^{\epsilon}}^{2} + 2\nu e^{-2\gamma t}\int_{0}^{t}e^{2\gamma s} \norm{y_{xt}^{\epsilon}(s)}^{2}ds&+ 6\delta e^{-2\gamma t}\int_{0}^{t}e^{2\gamma s}\norm{(y_{t}^{\epsilon}y^{\epsilon})(s)}^{2} ds\\[1mm]&+\frac{\nu}{\epsilon}e^{-2\gamma t}\int_{0}^{t}e^{2\gamma s}(y_{t}^{\epsilon}(s, 1))^{2} ds \leq C(\norm{y_{0}}_{1})e^{-2\gamma t}\norm{y_{0}}_{2}^{2}.
		\end{align*}
	\end{lemma}
	\begin{proof}
		Differentiating \eqref{p1} with respect to time $ t $ and forming $ L^{2}$-inner product with $ y_{t}^{\epsilon} $, we obtain
		\begin{align*}
			\frac{1}{2}\frac{d}{dt}\norm{y_{t}^{\epsilon}}^{2} + \nu \norm{y_{xt}^{\epsilon}}^{2}+ 3\delta\int_{0}^{1}(y_{t}^{\epsilon}y^{\epsilon})^{2}dx = \nu y_{xt}^{\epsilon}(t, 1)y_{t}(t, 1) + \alpha\norm{y_{t}^{\epsilon}}^{2}.
		\end{align*}
		Applying the penalized boundary condition and using the Young's inequality, we arrive at
		\begin{align*}
			\frac{1}{2}\frac{d}{dt}\norm{y_{t}^{\epsilon}}^{2} + \nu \norm{y_{xt}^{\epsilon}}^{2}+ 3\delta\int_{0}^{1}(y_{t}^{\epsilon}y^{\epsilon})^{2}dx + \frac{\nu}{2\epsilon}(y_{t}^{\epsilon}(t, 1))^{2}\leq (\frac{\nu r^{2}}{6\epsilon} + \alpha)\norm{y_{t}^{\epsilon}}^{2}.
		\end{align*}
		% Multiplying the above inequality by $ 2e^{2\gamma t} $, we have
		% \begin{align*}
		% 	\frac{d}{dt}(e^{\gamma t}\norm{y_{t}^{\epsilon}})^{2} + 2\nu e^{2\gamma t} \norm{y_{xt}^{\epsilon}}^{2}+ 6\delta e^{2\gamma t}\int_{0}^{1}(y_{t}^{\epsilon}y^{\epsilon})^{2}dx+ \frac{\nu}{\epsilon}e^{2\gamma t}(y_{t}^{\epsilon}(t, 1))^{2}\leq (2\gamma+ \frac{\nu r^{2}}{3\epsilon} + 2\alpha)e^{2\gamma t}\norm{y_{t}^{\epsilon}}^{2}.
		% \end{align*}
        Multiplying by $ 2e^{2\gamma t} $ and integrating with respect to time from $ 0 $ to $ t $, then using Lemmas \ref{L2.1} - \ref{L2.2}, we arrive at
		% Integrating with respect to time from $ 0 $ to $ t $ and multiplying the resulting inequality by $ e^{-2\gamma t} $, we use Lemmas \ref{L2.1} - \ref{L2.2} to obtain
		\begin{align}
			\label{2.112}
			\nonumber	\norm{y_{t}^{\epsilon}}^{2} &+ 2\nu e^{-2\gamma t}\int_{0}^{t}e^{2\gamma s} \norm{y_{xt}^{\epsilon}(s)}^{2}ds+ 6\delta e^{-2\gamma t}\int_{0}^{t}e^{2\gamma s}\left(\int_{0}^{1}(y_{t}^{\epsilon}y^{\epsilon})^{2}(s)dx\right)ds \\[1mm]& +\frac{\nu}{\epsilon}e^{-2\gamma t}\int_{0}^{t}e^{2\gamma s}(y_{t}^{\epsilon}(s, 1))^{2} ds \leq Ce^{-2\gamma t}\left(\norm{y_{0}}_{1}^{2}+ \norm{y_{0}}_{1}^{4}\right)+ e^{-2\gamma t}\norm{y_{t}^{\epsilon}(0)}^{2}.
		\end{align}
		Using \((a+b+c)^{2}\leq 3(a^{2}+b^{2}+c^{2})\) with Sobolev embedding, we can write from \eqref{p1} 
		\begin{align}
			\label{2.113}
			\norm{y_{t}^{\epsilon}}^{2}\leq 3\nu^{2}\norm{y_{xx}^{\epsilon}}^{2}+ 3\alpha^{2}\norm{y^{\epsilon}}^{2}+ 3\delta^{2}\norm{y^{\epsilon}}_{L^{6}}^{6}\leq 3\nu^{2}\norm{y_{xx}^{\epsilon}}^{2}+ 3\alpha^{2}\norm{y^{\epsilon}}^{2}+ C\norm{y^{\epsilon}}_{1}^{6} .
		\end{align}
		Using \eqref{2.113} at $ t=0 $, and substituting this bound into \eqref{2.112} completes the proof.
	\end{proof}

\begin{lemma}
		\label{L2.4}
		Let assumption $(A1)$ and the regularity result \eqref{r} hold. Then there exists a positive constant $C$ such that
		\begin{align*}
			\norm{y_{x}^{\epsilon}}^{2}+\frac{1}{\epsilon}(y^{\epsilon}(t, 1))^{2}+ 2\nu e^{-2\gamma t}\int_{0}^{t}e^{2\gamma s}\norm{y_{xx}^{\epsilon}(s)}^{2}ds&+ 6\delta e^{-2\gamma t}\int_{0}^{t} e^{2\gamma s}\norm{(y_{x}^{\epsilon} y^{\epsilon})(s)}^{2}ds\\[1mm]&+\frac{\alpha}{\epsilon}e^{-2\gamma t}\int_{0}^{t}e^{2\gamma s}(y^{\epsilon}(s, 1))^{4}ds\leq C(\norm{y_{0}}_{1}) e^{-2\gamma t}.
		\end{align*}
	\end{lemma}
	\begin{proof}
		Taking the $L^{2}$-inner product between \eqref{p1} and $-y_{xx}^{\epsilon}$, and integrating by parts, we obtain
		\begin{align}
			\label{2.114}
			\frac{1}{2}\frac{d}{dt}\norm{y_{x}^{\epsilon}}^{2}+ \nu \norm{y_{xx}^{\epsilon}}^{2}+3\delta\int_{0}^{1}(y_{x}^{\epsilon})^{2}(y^{\epsilon})^{2}dx&=\alpha\norm{y_{x}^{\epsilon}}^{2}- \alpha y^{\epsilon}(t, 1)y_{x}^{\epsilon}(1)\nonumber\\& \quad+ \delta(y^{\epsilon}(t, 1))^{3}y_{x}^{\epsilon}(1)+y_{t}^{\epsilon}(t, 1)y_{x}^{\epsilon}(1).
		\end{align}
		The second term on the right hand side of \eqref{2.114} is bounded by 
		\begin{align*}
			-\alpha y^{\epsilon}(t, 1)y_{x}^{\epsilon}(1)= \frac{\alpha}{\epsilon}(y^{\epsilon}(t, 1))^{2}+ \frac{\alpha}{\epsilon}y^{\epsilon}(t, 1) u^{\epsilon}(t)\leq \frac{3\alpha}{2\epsilon}(y^{\epsilon}(t, 1))^{2}+\frac{\alpha r^{2}}{6\epsilon}\norm{y^{\epsilon}}^{2}.
		\end{align*}
		On the right hand side of \eqref{2.114}, the third term is estimated by
		\begin{align*}
			\delta(y^{\epsilon}(t, 1))^{3}y_{x}^{\epsilon}(1)= -\frac{\delta}{\epsilon}(y^{\epsilon}(t, 1))^{4}+ \frac{\delta}{\epsilon}(y^{\epsilon}(t, 1))^{3}u^{\epsilon}(t)\leq-\frac{\delta}{2\epsilon}(y^{\epsilon}(t, 1))^{4}+\frac{\delta r^{2}}{6\epsilon}(y^{\epsilon}(t, 1))^{2}\norm{y^{\epsilon}}^{2}.
		\end{align*}
		The last term on the right of \eqref{2.114} can be written with the help of the Young's inequality
		\begin{align*}
			y_{t}^{\epsilon}(t, 1)y_{x}(1)\leq -\frac{1}{2\epsilon}\frac{d}{dt}(y^{\epsilon}(t, 1))^{2} +\frac{1}{2\epsilon}(y_{t}^{\epsilon}(t, 1))^{2}+ \frac{r^{2}}{6\epsilon}\norm{y^{\epsilon}}^{2}.
		\end{align*}
		Substituting the above estimates into \eqref{2.114}, we observe that
		\begin{align*}
			\frac{1}{2}\frac{d}{dt}\Big(\norm{y_{x}^{\epsilon}}^{2}&+\frac{1}{\epsilon}(y^{\epsilon}(t, 1))^{2}\Big) + \nu \norm{y_{xx}^{\epsilon}}^{2}+3\delta \int_{0}^{1}(y_{x}^{\epsilon})^{2}(y^{\epsilon})^{2}dx+ \frac{\delta}{2\epsilon}(y^{\epsilon}(t, 1))^{4}\\[1mm]&\leq \alpha\norm{y_{x}^{\epsilon}}^{2}+  \frac{ r^{2}}{6\epsilon}\left(\alpha + \delta (y^{\epsilon}(t, 1))^{2}\right) \norm{y^{\epsilon}}^{2}  \ \ + \frac{ r^{2}}{6\epsilon}\norm{y^{\epsilon}}^{2}+ \frac{1}{2\epsilon}(y_{t}^{\epsilon}(t, 1))^{2}+\frac{3\alpha}{2\epsilon}(y^{\epsilon}(t, 1))^{2}.
		\end{align*}
		% Multiplying the above inequality by $ 2e^{2\gamma t} $, it follows that
		% \begin{align*}
		% 	\frac{d}{dt}\left(e^{2\gamma t} \norm{y_{x}^{\epsilon}}^{2}+\frac{e^{2\gamma t}}{\epsilon}(y^{\epsilon}(t, 1))^{2}\right)&+ 2\nu e^{2\gamma t}\norm{y_{xx}^{\epsilon}}^{2}+ 6\delta e^{2\gamma t}\int_{0}^{1}(y_{x}^{\epsilon})^{2}(y^{\epsilon})^{2}dx+\frac{\delta}{\epsilon}e^{2\gamma t}(y^{\epsilon}(t, 1))^{4}\\[1mm]&\leq (2\gamma+2\alpha)e^{2\gamma t}\norm{y_{x}^{\epsilon}}^{2}+ \frac{ r^{2}}{6\epsilon}\left(\alpha + \delta (y^{\epsilon}(t, 1))^{2}\right)e^{2\gamma t} \norm{y^{\epsilon}}^{2}\\[1mm]& \ \ + \frac{e^{2\gamma t}}{\epsilon}(y_{t}^{\epsilon}(t, 1))^{2}+\frac{(3\alpha+ 2\gamma)}{\epsilon}e^{2\gamma t} (y^{\epsilon}(t, 1))^{2}+ \frac{ r^{2}}{6\epsilon}\norm{y^{\epsilon}}^{2}.
		% \end{align*}
		Multiplying by $ 2e^{2\gamma t} $, integrating over $[0,t]$, and multiplying the resulting inequality by $e^{-2\gamma t}$, we obtain the desired result by applying Lemmas \ref{L2.1} and \ref{L2.3}.
	\end{proof}
    
	\begin{lemma}
		\label{L2.5}
		Suppose that assumption $(A1)$ and the regularity result \eqref{r} hold. Then, there exists a positive constant $ C $ such that
		\begin{align*}
			\norm{y_{xt}^{\epsilon}}^{2}+ \nu e^{-2\gamma t}\int_{0}^{t}e^{2\gamma s}\norm{y_{xxt}^{\epsilon}(s)}^{2}ds+\frac{1}{\epsilon}(y_{t}^{\epsilon}(t, 1))^{2}\leq C(\norm{y_{0}}_{1})e^{-2\gamma t}\norm{y_{0}}_{3}^{2}.
		\end{align*}
	\end{lemma}
	\begin{proof}
		Differentiating \eqref{p1} with respect to time yields
		\begin{align}
			\label{2.10}
			y_{tt}^{\epsilon}= \nu y_{xxt}^{\epsilon}+ \alpha y_{t}^{\epsilon}-3\delta (y^{\epsilon})^{2} y_{t}^{\epsilon}.
		\end{align}
		Taking the $L^{2}$-inner product of \eqref{2.10} with $-y_{xxt}^{\epsilon}$ gives
		\begin{align*}
			\frac{1}{2}\frac{d}{dt}\norm{y_{xt}^{\epsilon}}^{2}+ \nu \norm{y_{xxt}^{\epsilon}}^{2}+\frac{1}{2\epsilon}\frac{d}{dt}\Big((y_{t}^{\epsilon}(t, 1))^{2}\Big) &=\frac{\alpha}{\epsilon}(y_{t}^{\epsilon}(t, 1))^{2}- \frac{\alpha}{\epsilon}u_{t}^{\epsilon}(t)y_{t}^{\epsilon}(t, 1)+ \alpha\norm{y_{xt}^{\epsilon}}^{2}\\[1mm]&\ \ + \frac{1}{\epsilon}\frac{d}{dt}\Big(y_{t}^{\epsilon}(t, 1) u_{t}^{\epsilon}(t)\Big)-  \frac{1}{\epsilon}y_{t}^{\epsilon}(t, 1)u_{tt}^{\epsilon}(t)+ 3\delta (y_{t}^{\epsilon}(y^{\epsilon})^{2}, y_{xxt}^{\epsilon}).
		\end{align*}
		Using $ u^{\epsilon}(t) = -r\int_{0}^{1} x y^{\epsilon}(t, x)\,dx $ and substituting $ y_{tt}^{\epsilon} $ from \eqref{2.10}, we obtain
		\begin{align*}
			\frac{1}{2}\frac{d}{dt}\norm{y_{xt}^{\epsilon}}^{2}+ \nu \norm{y_{xxt}^{\epsilon}}^{2}+\frac{1}{2\epsilon}\frac{d}{dt}\Big((y_{t}^{\epsilon})^{2}(t, 1)\Big) &=\frac{\alpha}{\epsilon}(y_{t}^{\epsilon})^{2}(t, 1)- \frac{\alpha}{\epsilon}u_{t}^{\epsilon}(t)y_{t}^{\epsilon}(t, 1)+ \alpha\norm{y_{xt}^{\epsilon}}^{2}\\[1mm]& \quad +\frac{1}{\epsilon}\frac{d}{dt}\Big(y_{t}^{\epsilon}(t, 1) u_{t}^{\epsilon}(t)\Big)+3\delta (y_{t}^{\epsilon}(y^{\epsilon})^{2}, y_{xxt}^{\epsilon})+I,
		\end{align*}
		where $ I=-\frac{1}{\epsilon}y_{t}^{\epsilon}(t, 1)\left( r\nu \int_{0}^{1}xy_{xxt}^{\epsilon}dx + r\alpha \int_{0}^{1}xy_{t}^{\epsilon}dx -3r \alpha \int_{0}^{1}xy_{t}^{\epsilon} (y^{\epsilon})^{2}dx\right).$
		
		Applying  Young's inequality and  the H\"older's inequality, we get
		\begin{align*}
			\frac{1}{2}\frac{d}{dt}\norm{y_{xt}^{\epsilon}}^{2}+ \nu \norm{y_{xxt}^{\epsilon}}^{2}&+\frac{1}{2\epsilon}\frac{d}{dt}((y_{t}^{\epsilon})^{2}(t, 1))-\frac{1}{\epsilon}\frac{d}{dt}(y_{t}^{\epsilon}(t, 1) u_{t}^{\epsilon}(t))\\[1mm]&\leq (\frac{3\alpha}{2\epsilon} + \frac{r^{2} \nu }{3} +\frac{\alpha \epsilon}{2}+ 3r^{2})(y_{t}^{\epsilon}(t, 1))^{2}+\frac{\alpha r^{2}}{3\epsilon}\norm{y_{t}^{\epsilon}}^{2} + \alpha \norm{y_{tx}^{\epsilon}}^{2} \\[1mm]& \quad + \frac{\nu}{2}\norm{y_{xxt}^{\epsilon}}^{2}+\frac{9\delta^{2}}{2\nu}\norm{y_{t}^{\epsilon}}_\infty^{2} \norm{y^{\epsilon}}_{L^{4}}^{4}.
		\end{align*}
		Multiplying both sides by $2e^{2\gamma t}$ and integrating over $[0, t]$, and using Sobolev embedding, we obtain
		\begin{align}
			\label{2.11}
			\nonumber	e^{2\gamma t} \norm{y_{xt}^{\epsilon}}^{2}+\nu \int_{0}^{t}e^{2\gamma s}\norm{y_{xxt}^{\epsilon}(s)}^{2}ds &+\frac{1}{\epsilon}e^{2\gamma t}(y_{t}^{\epsilon})^{2}(t, 1)\nonumber\\[1mm]&\leq \norm{y_{xt}^{\epsilon}(0)}^{2}+ \frac{1}{\epsilon}\Big(e^{2\gamma t}y_{t}^{\epsilon}(t, 1) u_{t}^{\epsilon}(t)\Big)\nonumber\\[1mm]& \quad +C\int_{0}^{t}e^{2\gamma s}\Big((y_{t}^{\epsilon})^{2}(s, 1)+\norm{y_{t}^{\epsilon}(s)}^{2} + \norm{y_{t}^{\epsilon}(s)}_{1}^{2} \norm{y^{\epsilon}(s)}_{L^{4}}^{4}\Big)ds.
		\end{align}
		After differentiating \eqref{p1} with respect to \(x\) and using \((a+b+c)^{2}\leq 3(a^{2}+b^{2}+c^{2})\) with Sobolev embedding, we can write 
		\begin{align}
			\label{2.12}
			\norm{y_{xt}^{\epsilon}}^{2}\leq C(\norm{y_{xxx}^{\epsilon}}^{2}+ \norm{y^{\epsilon}}_{1}^{2}+ \norm{y^{\epsilon}}_{1}^{6}).
		\end{align}
		Applying Young's inequality in \eqref{2.11}, substituting $t = 0$ in \eqref{2.12}, and using Lemmas \ref{L2.2}–\ref{L2.4}, we finally arrive at
		\begin{align*}
			e^{2\gamma t} \norm{y_{xt}^{\epsilon}}^{2}+\nu \int_{0}^{t}e^{2\gamma s}\norm{y_{xxt}^{\epsilon}(s)}^{2}ds +\frac{1}{\epsilon}e^{2\gamma t}(y_{t}^{\epsilon})^{2}(t, 1)\leq C(\norm{y_{0}}_{3}^{2}+ \norm{y_{0}}_{1}^{4}).
		\end{align*}
		Multiplying both sides by $e^{-2\gamma t}$ completes the proof.
	\end{proof}
	\begin{lemma}
		\label{L2.26}
		Under assumption $(A1)$ and the regularity result \eqref{r}, the following estimate holds
		\begin{align*}
			\nu\norm{y_{xx}^{\epsilon}}\leq C(\norm{y_{0}}_{2})e^{-2\gamma t}.
		\end{align*}
	\end{lemma}
	\begin{proof}
		% From \eqref{p1}, we have
		% \begin{align*}
		% 	\nu y_{xx}^{\epsilon}= y_{t}^{\epsilon}-\alpha y^{\epsilon}+\delta (y^{\epsilon})^{3}.
		% \end{align*}
		Taking the $L^2$-norm and applying the triangle inequality with Sobolev embedding, From \eqref{p1} we obtain 
		%	\begin{align*}
			%	 \nu\norm{y_{xx}^{\epsilon}}^{2}\leq C(\norm{y_{t}^{\epsilon}}^{2}+ \norm{y^{\epsilon}}^{2}+\norm{y^{\epsilon}}_{L^{6}}^{6}).
			%	\end{align*}
		%	Using Sobolev embedding, we get
		\begin{align*}
			\nu\norm{y_{xx}^{\epsilon}}^{2}\leq C(\norm{y_{t}^{\epsilon}}^{2}+ \norm{y^{\epsilon}}^{2}+\norm{y^{\epsilon}}_{1}^{6}).
		\end{align*}
		The result follows from Lemmas \ref{L2.1}-\ref{L2.3}, completing the proof.
	\end{proof}
	The proof of the following theorem follows from Lemmas \ref{L2.1} - \ref{L2.26}.
	\begin{theorem}
		\label{th2.21}
		Let  assumption $(A1)$ and the regularity result \eqref{r} hold. Then, there exists a positive constant $C$ such that
		\begin{align*}
			\norm{y^{\epsilon}}_{2}^{2}+\norm{y_{t}^{\epsilon}}_{1}^{2}+(y_{t}^{\epsilon}(t,1))^{2}&+\norm{y^{\epsilon}}_{L^{4}}^{4}+(y^{\epsilon}(t,1))^{2}+e^{-2\gamma t}\int_{0}^{t}e^{2\gamma s}\left(\norm{y^{\epsilon}_{t}(s)}_{2}^{2}+\norm{y^{\epsilon}_{t}(s)}^{2}+\norm{y^{\epsilon}(s)}_{L^{4}}^{4}\right)ds\\[1mm]&+e^{-2\gamma t}\int_{0}^{t}e^{2\gamma s}\Big((y^{\epsilon}(s,1))^{2}+(y_{t}^{\epsilon}(s,1))^{2}+(y^{\epsilon}(s,1))^{4} \Big)ds\leq e^{-2\gamma t}C(\norm{y_{0}}_{3}).
		\end{align*}
	\end{theorem}
	
	\section{Continuous Dependence Property and Convergence Analysis. }
	In this section, we establish the continuous dependence property of the penalized feedback control problem \eqref{p1}-\eqref{p2}, which holds under assumption $(A1)$ and the regularity result \eqref{r}.
	
	Let $ y_{1}^{\epsilon} $ and $ y_{2}^{\epsilon} $ be two solutions of the problem \eqref{p1} - \eqref{p2}. Denote $ E= y_{1}^{\epsilon}- y_{2}^{\epsilon}. $ Then, subtracting the equations satisfied by $y_{2}^{\epsilon}$ from those of $y_{1}^{\epsilon}$, we obtain
	\begin{align}
		\label{2.118}
		E_{t}&=\nu E_{xx}+ \alpha E -\delta ((y_{1}^{\epsilon})^{3}- (y_{2}^{\epsilon})^{3}),\\
		\nonumber E_{x}(t,1)&=-\frac{1}{\epsilon}E(t,1)-\frac{r}{\epsilon}\int_{0}^{1}xE(t,x)dx,\quad E(t,0)=0,
		\\\nonumber E(0, x)&=0.	
	\end{align}
	Taking the \(L^{2}\)-inner product between \(E\) and \(\eqref{2.118}\), we get
	\begin{align*}
		\frac{1}{2}\frac{d}{dt}\norm{E}^{2} + \nu \norm{E_{x}}^{2} =\nu E_{x}(t,1)E(t,1) + \alpha \norm{E}^{2}-\frac{\delta}{2} \int_{0}^{1}E^{2}\Big((y_{1}^{\epsilon})^{2}+(y_{2}^{\epsilon})^{2}+(y_{1}^{\epsilon}+y_{2}^{\epsilon} )^{2}\Big)dx.
	\end{align*}
   Since the last term on the right-hand side is non-positive, following the proof of Lemma \ref{L2.1}, we get 
	% Since the last term on the right-hand side is non-positive, we estimate the boundary term using Young’s inequality
	% \begin{align*}
	% 	\frac{1}{2}\frac{d}{dt}\norm{E}^{2} + \nu \norm{E_{x}}^{2}+ \frac{7\nu}{8\epsilon}(E(t, 1))^{2}\leq \alpha \norm{E}^{2}+ \frac{2\nu r^{2}}{3\epsilon}\norm{E}^{2}.
	% \end{align*}
	% Multiplying both sides by \(2e^{2\gamma t}\) and using the Poincar\'e inequality along with assumption $(A1)$, we obtain
	% \begin{align*}
	% 	\frac{d}{dt}(\norm{E}^{2} e^{2\gamma t}) + \beta_{1}  e^{2\gamma t} \norm{E_{x}}^{2}+ \frac{7\nu}{4\epsilon} e^{2\gamma t}(E(t, 1))^{2}\leq 0 ,
	% \end{align*}
	% where $ 0\leq\beta_{1}=2\nu-\gamma-\alpha-\frac{2\nu r^{2}}{3\epsilon}. $
	
	% Integrating with respect to time from $ 0 $ to $ t $ and multiplying the result by \(e^{-2\gamma t}\), we obtain
	\begin{align}
		\label{2.13}
		\norm{E}^{2}+ \beta_{1} e^{-2\gamma t}\int_{0}^{t}e^{2\gamma s}\norm{E_{x}(s)}^{2}ds +\frac{\nu}{\epsilon}e^{-2\gamma t}\int_{0}^{t}e^{2\gamma s}(E(s, 1))^{2}ds \leq e^{-2\gamma t}\norm{E_{0}}^{2},
	\end{align}
    where $ 0\leq\beta_{1}=2\nu-\gamma-\alpha-\frac{2\nu r^{2}}{3\epsilon}. $
	Since \(E(0) = 0\), it follows that \(\norm{E(t)} = 0\) for all \(t\geq 0\). Hence, we conclude that \(y_{1}^{\epsilon} = y_{2}^{\epsilon}\), and inequality \eqref{2.13} ensures the continuous dependence of the solution on the initial data.

	\subsection{Convergence Analysis.}
	\label{sbs4.1}
	In this subsection, we analyze the convergence analysis of the solution $y^{\epsilon}$ of the penalized problem \eqref{p1}-\eqref{p2} to the solution $y$ of the corresponding closed loop system \eqref{ch1}-\eqref{c}, as the penalty parameter $\epsilon \rightarrow 0$. We define the difference $z:= y^{\epsilon} - y$.
	
	By subtracting \eqref{p1}-\eqref{p2} from \eqref{ch1}-\eqref{c}, we obtain
	\begin{align}
		\label{2.14}
		z_{t}= \nu z_{xx}+\alpha z - \delta ((y^{\epsilon})^{3}- y^{3}),\\
		\epsilon\left(\frac{\partial z}{\partial x} (1)+\frac{\partial y}{\partial x}(1) \right)+ z(1)= (u^{\epsilon}(t)-u(t)),\ 
		z(t,0)=0,\\
		z(0,x)=0.
		\label{2.17}
	\end{align}
	The weak formulation associated with \eqref{2.14}-\eqref{2.17} seeks $z \in H_{\{0\}}^{1}$ such that
	\begin{align}\label{2.21}
		(z_{t}, \chi) + \nu (z_{x}, \chi_{x})+ \frac{\nu}{\epsilon} z(t, 1)\chi(1)+ \delta ((y^{\epsilon})^{3}- y^{3}, \chi)&= \frac{\nu}{\epsilon}( u^{\epsilon}(t) -u(t))\chi(1)- \nu \frac{\partial y}{\partial x}(t, 1) \chi(1)\nonumber\\[1mm]& \ \ + \alpha(z, \chi), \  \forall \chi\in H_{\{0\}}^{1}.
	\end{align}
	The following theorem guarantees the convergence analysis in the \(L^{2}\)-norm for the state variable.
	\begin{theorem}
		\label{th2.1}
		Let $z$ be the solution of \eqref{2.14}-\eqref{2.17}. Under assumption \((A1)\) and the regularity result \eqref{r}, there exists a positive constant $C=C(\norm{y_{0}}_{1})$ independent of $\epsilon$ such that
		\begin{align*}
			\norm{z}^{2}+ e^{-2\gamma t}\beta_{1} \int_{0}^{t}e^{2\gamma s}\norm{z_{x}(s)}^{2}ds + \frac{3\nu e^{-2\gamma t}}{2\epsilon}\int_{0}^{t}e^{2\gamma s}(z(s,1))^{2}ds \leq Ce^{-2\gamma t} \epsilon,
		\end{align*}
		where $ 0<\beta_{1}=\nu-\gamma-\alpha. $
	\end{theorem}
	\begin{proof}
		Set \(\chi = z\) in the weak formulation \eqref{2.21} to obtain
		\begin{align*}
			\frac{1}{2}\frac{d}{dt}\norm{z}^{2} + \nu \norm{z_{x}}^{2} + \frac{\nu}{\epsilon}(z(t,1))^{2} \leq \alpha\norm{z}^{2}+ \frac{\nu}{\epsilon}\left(u^{\epsilon}(t)-u(t)\right) z(t,1)- \nu \frac{\partial y}{\partial x}(t,1)z(t,1).
		\end{align*}
		Since $ u^{\epsilon}(t) -u(t) = -r\int_{0}^{1}xz(t, x)dx, $ therefore, by the H\"older's inequality, we obtain
		\begin{align}
			\label{2.23}
			(u^{\epsilon}(t) -u(t))^{2}\leq \frac{r^{2}}{3}\norm{z}^{2}.
		\end{align}
		Using Young's inequality on the right hand side of the above inequality, we arrive at
		\begin{align*}
			\frac{1}{2}\frac{d}{dt}\norm{z}^{2} + \nu \norm{z_{x}}^{2} + \frac{3\nu}{4\epsilon}(z(t,1))^{2} \leq (\alpha+\frac{2\nu r^{2}}{3\epsilon})\norm{z}^{2} + 2\epsilon \nu (\frac{\partial y}{\partial x}(t,1))^{2}.
		\end{align*}
		Multiplying by \(2e^{2\gamma t}\) to the above inequality yields
		\begin{align*}
			\frac{d}{dt}(e^{2\gamma t}\norm{z}^{2})+ 2\nu e^{2\gamma t}\norm{z_{x}}^{2} + \frac{3\nu e^{2\gamma t}}{2\epsilon}(z(t,1))^{2} \leq 2e^{2\gamma t}(\gamma+\alpha+\frac{2\nu r^{2}}{3\epsilon})\norm{z}^{2} + 4\epsilon \nu e^{2\gamma t} (\frac{\partial y}{\partial x}(t,1))^{2}.
		\end{align*}
		Using the Poincar\'e inequality with assumption in Lemma \(\ref{L2.1}\), and integrating with respect to time from  \( 0 \) to \(t \), we observe that
		\begin{align}
			\label{2.18}
			e^{2\gamma t}\norm{z}^{2}+ \beta_{1} \int_{0}^{t}e^{2\gamma s}\norm{z_{x}(s)}^{2}ds + \frac{3\nu}{2\epsilon}\int_{0}^{t}e^{2\gamma s}(z(s,1))^{2}ds \leq 2\epsilon \nu \norm{e^{\gamma t}\frac{\partial y}{\partial x}(t,1)}_{L^{2}(0,t)}^{2},
		\end{align}
		where $ 0<\beta_{1}=\nu-\gamma -\alpha. $
		
		Now, we establish the bound of $ \frac{\partial y}{\partial x}(t,1).$
		Using integration by parts, we can write
		\begin{align*}
			\frac{\partial y}{\partial x}(t,1)=\int_{0}^{1}\frac{\partial y}{\partial x}(t,x)dx + \int_{0}^{1}x\frac{\partial^{2} y}{\partial x^{2}}(t,x)dx.
		\end{align*}
		A use of the H\"older's inequality with $ (a+b)^{2}\leq 2 (a^{2}+ b^{2}) $ to the above equation, we obtain
		\begin{align*}
			(\frac{\partial y}{\partial x}(t,1))^{2}\leq 2\int_{0}^{1}(\frac{\partial y}{\partial x})^{2}dx +\frac{2}{3}\int_{0}^{1}(\frac{\partial^{2} y}{\partial x^{2}}(t,x))^{2}dx.
		\end{align*}
		Multiplying by \(e^{2\gamma t}\) and integrating with respect to time from $ 0 $ to $ t $ to the above inequality, it follows that
		\begin{align}
			\label{2.19}
			\int_{0}^{t}e^{2\gamma s}(\frac{\partial y}{\partial x}(s,1))^{2}ds\leq C\int_{0}^{t}e^{2\gamma s}\norm{y(s)}_{2}^{2}ds.
		\end{align}
		Using \eqref{2.18} and \eqref{2.19},  we obtain 
		\begin{align}
			\label{2.20}
			e^{2\gamma t}\norm{z}^{2}+ \beta_{1} \int_{0}^{t}e^{2\gamma s}\norm{z_{x}(s)}^{2}ds + \frac{3\nu}{2\epsilon}\int_{0}^{t}e^{2\gamma s}(z(s,1))^{2}ds \leq C \epsilon\int_{0}^{t}e^{2\gamma s}\norm{y(s)}_{2}^{2}ds,
		\end{align}
		where $ C=C(\nu)>0.$
		This proof is completed with Remark \ref{r1.2} and Lemma \ref{L1.1}, post-multiplying by \(e^{-2\gamma t}\) to the above equation.
	\end{proof}
	In the following lemma, we establish the convergence analysis of the control input and the state variable in the \(H^{1}\)-norm with respect to the penalty parameter.
	\begin{lemma}
		\label{L2.6}
		Suppose that the hypothesis of  Theorem \ref{th2.1}, assumption $(A1)$, and the regularity result \eqref{r}  hold, then
		\[
		\nu (3\epsilon-r^{2}) \norm{z_{x}}^{2}+ \frac{3}{2}\nu z^{2}(t, 1)+3\epsilon e^{-2\gamma t} \int_{0}^{t}e^{2\gamma s}\norm{z_{t}}^{2}ds\leq C(\norm{y_{0}}_{3})e^{-2\gamma t}(\epsilon + \epsilon^{2}).
		\]
	\end{lemma}
	\begin{proof}
		Set $ \chi = z_{t} $ in \eqref{2.21} to get
		%	\begin{align*}
			%		\frac{\nu}{2}\frac{d}{dt}(\norm{z_{x}}^{2})+ \norm{z_{t}}^{2}+ \frac{\nu}{2\epsilon}\frac{d}{dt}(z^{2}(t,1))+ \delta((y^{\epsilon})^{3}- y^{3}, z_{t})&=\frac{\nu}{\epsilon}( u^{\epsilon}(t) -u(t))z_{t}(t, 1)\\&- \nu \frac{\partial y}{\partial x}(t, 1) z_{t}(t, 1)+ \alpha(z, z_{t}).
			%	\end{align*}
		%	Rewrite to the above equation in the form 
		\begin{align}
			\label{2.22}
			\frac{\nu}{2}\frac{d}{dt}(\norm{z_{x}}^{2})+ \norm{z_{t}}^{2}&+ \frac{\nu}{2\epsilon}\frac{d}{dt}(z^{2}(t,1))\nonumber\\[1mm]&= - \delta\Big((y^{\epsilon})^{3}- y^{3}, z_{t}\Big)+ \frac{\nu}{\epsilon}\frac{d}{dt}\Big(\left( u^{\epsilon}(t) -u(t)\right)z(t, 1)\Big)-\frac{\nu}{\epsilon}( u^{\epsilon}(t) -u(t))_{t}z(t, 1)\nonumber\\[1mm]&\ \ -\nu\frac{d}{dt}\Big(\frac{\partial y}{\partial x}(t, 1) z(t, 1)\Big)+\nu \frac{d}{dt}\Big(\frac{\partial y}{\partial x}(t, 1)\Big) z(t, 1)+ \alpha(z, z_{t}).
		\end{align}
		Using the Sobolev embedding, Lemma \ref{L1.1}  and Lemmas \ref{L2.1} -  \ref{L2.2} with $\gamma =0$, the first term  on the right hand side of \eqref{2.22} is bounded by 
		\begin{align*}
			\delta((y^{\epsilon})^{3}- y^{3}, z_{t})\leq C\delta \norm{z} (\norm{y^{\epsilon}}_{\infty}^{2}+\norm{y}_{\infty}^{2})\norm{z_{t}} \leq C\norm{z}^{2}\norm{y}_{1}^{2}+C\norm{z}^{2}+ \frac{1}{6}\norm{z_{t}}^{2}.
		\end{align*}
		A use of the Young's inequality with \eqref{2.23}, the third term on the right hand side of \eqref{2.22} can be written as 
		\begin{align*}
			\frac{\nu}{\epsilon}( u^{\epsilon}(t) -u(t))_{t}z(t, 1)\leq \frac{1}{6}\norm{z_{t}}^{2}+ \frac{\nu^{2}r^{2}}{2\epsilon^{2}}z^{2}(t, 1).
		\end{align*}
		Substituting these values into \eqref{2.22} with the use of Young's inequality, we arrive at
		\begin{align*}
			\frac{\nu}{2}\frac{d}{dt}(\norm{z_{x}}^{2}) +\frac{1}{2} \norm{z_{t}}^{2}+ \frac{\nu}{2\epsilon}\frac{d}{dt}(z^{2}(t,1))&\leq C\norm{z}^{2}\norm{y}_{1}^{2} +  \frac{\nu}{\epsilon}\frac{d}{dt}\Big(( u^{\epsilon}(t) -u(t))z(t, 1)\Big)-\nu\frac{d}{dt}\Big(\frac{\partial y}{\partial x}(t, 1) z(t, 1)\Big)\\[1mm]& \ \ +(\frac{\nu^{2}r^{2}}{2\epsilon^{2}} + \frac{\nu^{2}}{\epsilon^{2}})z^{2}(t, 1)+\frac{\epsilon^{2}}{4}(\frac{\partial^{2}y}{\partial x \partial t}(t, 1))^{2}+C\norm{z}^{2}.
		\end{align*}
		Multiplying by \(2e^{2\gamma t}\) and integrating with respect to time from $ 0 $ to $ t $, and using Young's inequality in the resulting inequality with Lemma \ref{L1.1} yields
		\begin{align*}
			\nu e^{2\gamma t}\Big(\norm{z_{x}}^{2}+ \frac{1}{2\epsilon}z^{2}(t,1)\Big)+ \int_{0}^{t}e^{2\gamma s}\norm{z_{t}(s)}^{2}ds &\leq 
			C e^{2\gamma t}\Big(\epsilon(\frac{\partial y}{\partial x}(t, 1))^{2}+\norm{z}^{2}\Big)\\[1mm]&\quad+ C\int_{0}^{t}e^{2\gamma s}\Big(\norm{z(s)}^{2}+ z^{2}(s, 1)+(\frac{\partial y}{\partial x}(s, 1))^{2}\Big)ds\\[1mm]&\quad + \epsilon^{2}\int_{0}^{t}e^{2\gamma s}(\frac{\partial^{2}y}{\partial x \partial t}(s, 1))^{2} ds+ 2\gamma\int_{0}^{t}e^{2\gamma s}\norm{z_{x}(s)}^{2}ds.
		\end{align*}
		Since 
		\[
		(\frac{\partial^{2}y}{\partial x \partial t}(t, 1))^{2}\leq 2\norm{y_{xt}}^{2}+ \frac{2}{3}\norm{y_{xxt}}^{2},
		\]
		therefore, applying Theorems \ref{th2.1} and \ref{th2.2}, and Remark \ref{r1.2} - \ref{r1.3} completes the rest of the proof.
	\end{proof}
	\begin{remark}
		\label{r2.3}
		From  Lemma \ref{L2.6} with \(\gamma=0\), we arrive at
		\[ (z(t,1))^{2} \leq C (\epsilon+\epsilon^{2}).\]
		Therefore, it follows that
		\[y^{\epsilon}(t, 1)\rightarrow y(t, 1), \ \text{when} \quad \epsilon\rightarrow 0.\]
		Hence, we can write
		\[u^{\epsilon}(t) \rightarrow u(t), \  \text{as} \quad \epsilon\rightarrow 0. \]
	\end{remark}
	
	\section{Finite Element Approximation.}
	In this section, we present the semi-discrete approximation of the problem using the finite element method, keeping the time variable continuous.
	
	Let $N_{h}$ be a positive integer, and consider a partition $P = \{ 0 = x_{0} < x_{1} < \ldots < x_{N_{h}} = 1 \}$ of the interval $(0, 1)$ into subintervals $I_{i} = (x_{i-1}, x_{i})$ for $1 \leq i \leq N_{h}$, where $h_{i} = x_{i} - x_{i-1}$. Let the mesh parameter be defined by $h = \max_{1 \leq i \leq N_{h}} h_{i}$.
	We now define a finite-dimensional subspace $S_{h} \subset H_{\{0\}}^{1}$ as follows:
	\[S_{h}=\{ \phi_{h}\in C^{0}((0, 1)): \phi_{h}|_{I_{i}} \in P_{1}(I_{i}) \ \ 1\leq i\leq N_{h},\;\phi_h(0)=0\}, \]
	where $P_{1}(I_{i})$ denotes the space of polynomials of degree at most one on each subinterval $(I_{i})_{i=1}^{N_{h}}$.
	
	The semi-discrete finite element scheme corresponding to problem \eqref{1.12} is to find $y_{h}^{\epsilon}(t,x) \in S_{h}$ such that
	\begin{align}
		\label{3.1}
		(y_{ht}^{\epsilon}, \chi)+\nu (y_{hx}^{\epsilon}, \chi_{x}) + \frac{\nu}{\epsilon}y_{h}^{\epsilon}(t,1) \chi(1)+ \delta((y_{h}^{\epsilon})^{3}, \chi)= \frac{\nu}{\epsilon}u_{h}^{\epsilon}(t)\chi(1)+\alpha(y_{h}^{\epsilon}, \chi)  \medspace \ \text{for all} \ \chi\in S_{h},
	\end{align}
	with the initial condition $y_{h}^{\epsilon}(0, x) = y_{0h}$, where $y_{0h}$ is an approximation of $y_{0}$ given by $y_{0h} = P_{h} y_{0}$. Here, $P_{h}$ denotes the projection operator from $H_{\{0\}}^{1}$ onto $S_{h}$.
	
	Given that $S_{h}$ is a finite-dimensional subspace of $H_{\{0\}}^{1}$, the semi-discrete problem \eqref{3.1} defines a system of nonlinear ordinary differential equation. By Picard’s theorem, the existence and uniqueness of a solution hold locally, i.e., there exists an interval $(0, t^{*})$ such that $y_{h}^{\epsilon}(t,x)$ exists for all $t \in (0, t^{*})$. To extend this to global existence for all $t > 0$, we apply the following result.
	
	\begin{lemma}
		\label{L3.1}
		Let assumption $(A1)$ and the regularity result \eqref{r} hold. Then, the following estimate holds
		\begin{align*}
			\norm{y_{h}^{\epsilon}}^{2}+ \beta e^{-2\gamma t}\int_{0}^{t}e^{2\gamma s}\Big(\norm{y_{hx}^{\epsilon}(s)}^{2} + (y_{h}^{\epsilon}(s,1))^{2}\Big)ds+ 2\delta e^{-2\gamma t}\int_{0}^{t} e^{2\gamma s}\norm{y_{h}^{\epsilon}(s)}_{L^{4}}^{4}ds\leq e^{-2\gamma t}\norm{y_{0h}}^{2},
		\end{align*}
		where $ \beta>0. $
	\end{lemma}
	\begin{proof}
		The proof follows directly from Lemma \ref{L2.1}.
	\end{proof}
	\subsection{Error estimates.}
	In this subsection, we establish the error estimates of the semi-discrete scheme \eqref{3.1} for the penalized feedback control problem of the state variable and control input.
	
	Define the auxiliary projection $ \tilde{y}_{h}^{\epsilon}\in S_{h} $ of $ y^{\epsilon} $ through the following form
	\begin{align}
		\label{3.2}
		(y_{x}^{\epsilon}-\tilde{y}_{hx}^{\epsilon}, \chi_{x})=0, \quad \forall  \ \chi\in S_{h}.
	\end{align}
	The existence and uniqueness of $\tilde{y}_{h}^{\epsilon}$ follow from the Lax-Milgram lemma.
	Let $\eta := y^{\epsilon} - \tilde{y}_{h}^{\epsilon}$ denotes the error associated with the auxiliary projection. For more details on this type of projection, see \cite{MR3790146, MR1225705}. Then the following error estimates hold
	\begin{align}
		\label{3.3}
		\norm{\eta}\leq Ch^{2}\norm{y^{\epsilon}}_{2} , \  \text{and} \ \ \norm{\eta_{t}}\leq Ch^{2}\norm{y_{t}^{\epsilon}}_{2}.
	\end{align} 
	For a detailed proof of \eqref{3.3}, refer to \cite{MR0534334}. 
	Additionally, we require estimates of $\eta$ at $x=1$ for the subsequent error analysis. The following result is proved in \cite{MR3790146, MR1225705}.
	\begin{lemma}
		\label{L3.2}
		At $x = 1$, the following estimates hold
		\begin{align}
			|\eta(x)|\leq Ch^{2}\norm{y^{\epsilon}}_{2}, \quad \text{and}\ \ 	|\eta_{t}(x)|\leq Ch^{2}\norm{y_{t}^{\epsilon}}_{2}.
		\end{align}
	\end{lemma}
	Define the error 
	\[ e:=y^{\epsilon}-y_{h}^{\epsilon}= (y^{\epsilon}-\tilde{y}_{h}^{\epsilon})-(y_{h}^{\epsilon}-\tilde{y}_{h}^{\epsilon})=: \eta - \theta,\]
	where $ \eta= (y^{\epsilon}-\tilde{y}_{h}^{\epsilon})  $ and $ \theta = (y_{h}^{\epsilon}-\tilde{y}_{h}^{\epsilon}).$
	
	Choose $\tilde{y}_{h}^{\epsilon}(0) = y_{0h}^{\epsilon}$ so that $\theta(0) = 0$. It is sufficient to estimate $\theta$, since the bounds for $\eta$ and $\eta_{t}$ are provided in Lemma \ref{L3.2} and equation \eqref{3.3}.
	
	Subtracting \eqref{3.1} from \eqref{1.12} and using \eqref{3.2}, we obtain
	\begin{align}
		\label{3.5}
		(\theta_{t}, \chi) + \nu (\theta_{x}, \chi_{x})+ \frac{\nu}{\epsilon}\theta(t, 1)\chi(1)-\alpha(\theta, \chi)&= (\eta_{t}, \chi)+ \frac{\nu}{\epsilon}\eta(t, 1)\chi(1) + \delta((y^{\epsilon})^{3}- (y_{h}^{\epsilon})^{3}, \chi)\nonumber\\[1mm]& \ \ -\frac{\nu}{\epsilon}(u^{\epsilon}(t)-u_{h}^{\epsilon}(t)) \chi(1) - \alpha (\eta, \chi),
	\end{align}
	where   
	\begin{align*}
		(y^{\epsilon})^{3}- (y_{h}^{\epsilon})^{3}= \eta^{3}- \theta^{3}+ 3\eta y^{\epsilon} (y^{\epsilon}-\eta) - 3\theta y_{h}^{\epsilon}(y_{h}^{\epsilon}-\theta),
	\end{align*}
	and $u^{\epsilon}(t)-u_{h}^{\epsilon}(t)= -r\int_{0}^{1}x(\eta - \theta)dx.$
	
	In the following lemma, we estimate $\norm{\theta}$ .
	\begin{lemma}
		\label{L3.3}
		Under assumption $(A1)$ and the regularity result \eqref{r}, there exists a decay rate $ 0< \gamma \leq \nu - \alpha - \frac{\alpha}{M} $ and $\frac{\alpha}{\nu} \leq \frac{2M}{1+M} \left(\frac{3\epsilon - r^{2}}{3\epsilon}\right)$, and a positive constant $ C =C(\norm{y_{0}}_{3}, \norm{y_{0}}_{1}) $ independent of \(h\) and \(\epsilon\), such that
		\begin{align*}
			\norm{\theta}^{2}+ \beta_{2} e^{-2\gamma t} \int_{0}^{t}e^{2\gamma s} \norm{\theta_{x}(s)}^{2} ds &+\frac{\nu}{\epsilon}e^{-2\gamma t} \int_{0}^{t}e^{2\gamma s}\theta^{2}(s, 1) ds +\delta e^{-2\gamma t}\int_{0}^{t}e^{2\gamma s}\norm{\theta(s)}_{L^{4}}^{4}ds\\[1mm]&\leq Ch^{4}(1+\frac{1}{\epsilon})e^{-2\gamma t},
		\end{align*}
		where \(\beta_{2}=(\nu -\gamma-\alpha -\frac{\alpha}{M})> 0\), and \(M>0\) is a  large constant.
	\end{lemma}
	\begin{proof}
		Select $ \chi= \theta $ in \eqref{3.5} to get
		\begin{align}
			\label{3.6}
			\nonumber	\frac{1}{2}\frac{d}{dt}\norm{\theta}^{2}+\nu \norm{\theta_{x}}^{2}&+ \frac{\nu}{\epsilon}\theta^{2}(t, 1)-\alpha\norm{\theta}^{2}+ \delta \norm{\theta}_{L^{4}}^{4}	\nonumber\\[1mm]&=\left((\eta_{t}, \theta)-\alpha(\eta, \theta)\right)+\frac{\nu}{\epsilon}\eta(t, 1)\theta(t, 1)+ \delta(\eta^{3}, \theta)- 3\delta(\eta^{2}y^{\epsilon}, \theta)	\nonumber\\[1mm]& \ \ + 3\delta(\eta (y^{\epsilon})^{2}, \theta)- 3\delta(\theta(y_{h}^{\epsilon})^{2}, \theta)+3\delta(\theta^{2}y_{h}^{\epsilon}, \theta)+\frac{r\nu}{\epsilon}\left(\int_{0}^{1}x(\eta -\theta)dx\right)\theta(t, 1),	\nonumber\\[1mm]&=:\sum_{i=1}^{5}I_{i}(\theta).
		\end{align} 
		The initial term  $ I_{1}(\theta) $ on the right hand side of \eqref{3.6} is estimated by
		\begin{align*}
			I_{1}(\theta)=\left((\eta_{t}, \theta)-\alpha(\eta, \theta)\right)\leq C\norm{\eta_{t}}^{2}+ C\norm{\eta}^{2}+ \frac{\alpha}{2M}\norm{\theta}^{2},
		\end{align*}
		where \( M \) is a large positive constant.
		
		The second term $ I_{2}(\theta) $ on the right hand side of \eqref{3.6} is bounded by
		\begin{align*}
			I_{2}(\theta) =\alpha(\eta^{3}, \theta)+\frac{\nu}{\epsilon}\eta(t, 1)\theta(t, 1)\leq C\norm{\eta}_{L^{4}}^{4}+ \frac{\delta}{4}\norm{\theta}_{L^{4}}^{4}+ \frac{C}{\epsilon}\eta^{2}(t, 1)+ \frac{\nu}{4\epsilon}\theta^{2}(t, 1).
		\end{align*}
		%The third term  $ I_{3}(\theta) $ on the right hand side of \eqref{3.6} is bounded by
		%\begin{align*}
		%	 I_{3}(\theta)=\alpha(\eta^{3}, \theta)\leq C\norm{\eta}^{4}+ \frac{\delta}{4}\norm{\theta}^{4}.
		%\end{align*}
		A use of Young's inequality to  $ I_{3}(\theta) $ on the right hand side of \eqref{3.6} gives
		\begin{align*}
			I_{3}(\theta)=-3\delta(\eta^{2}y^{\epsilon}, \theta)+ 3\delta(\eta (y^{\epsilon})^{2}, \theta)\leq C\norm{\eta}_{L^{4}}^{4}\norm{y^{\epsilon}}_{\infty}^{2}+ \frac{\alpha}{2M}\norm{\theta}^{2}+ C\norm{\eta}^{2}\norm{y^{\epsilon}}_{\infty}^{4}.
		\end{align*}
		Applying $ y_{h}^{\epsilon} =\theta+\tilde{y}_{h}^{\epsilon}$ and using the Young's inequality with $ \norm{\tilde{y}_{h}^{\epsilon}}_{\infty}\leq C\norm{y^{\epsilon}}_{1} $ in  $ I_{4}(\theta) $ on the right hand side of \eqref{3.6}, we arrive at
		\begin{align*}
			I_{4}(\theta)=-3\delta(\theta (\tilde{y}_{h}^{\epsilon})^{2}, \theta)- 3\delta(\theta^{2}\tilde{y}_{h}^{\epsilon}, \theta)\leq C\norm{y^{\epsilon}}_{1}^{2}\norm{\theta}^{2}+\frac{\delta}{4}\norm{\theta}_{L^{4}}^{4}.
		\end{align*}
		Finally, the last term $ I_{5}(\theta) $ on the right hand side of \eqref{3.6} with the use of the H\"older's inequality and Young's inequality is bounded by
		\begin{align*}
			I_{5}(\theta)\leq \frac{C}{\epsilon} \norm{\eta}^{2}+ \frac{2\nu r^{2}}{3\epsilon}\norm{\theta}^{2}+\frac{\nu}{4\epsilon}\theta^{2}(t, 1).
		\end{align*}
		Substituting the estimates of \( (I_i(\theta))_{i=1}^{5} \) into \eqref{3.6} and applying Lemma \ref{L2.1} with  Sobolev embedding, we obtain
		\begin{align*}
			\frac{1}{2}\frac{d}{dt}\norm{\theta}^{2}+\nu \norm{\theta_{x}}^{2}&+ \frac{\nu}{2\epsilon}\theta^{2}(t, 1)+\frac{\delta}{2}\norm{\theta}_{L^{4}}^{4}\\[1mm]&\leq C\Big(\norm{\eta_{t}}^{2}+C(1+\frac{1}{\epsilon}) \norm{\eta}^{2}+\norm{\eta}_{L^{4}}^{4} +\frac{C}{\epsilon} \eta^{2}(t, 1)\Big)\\[1mm]&\quad+ (\alpha +\frac{\alpha}{M}+ \frac{2\nu r^{2}}{3\epsilon})\norm{\theta}^{2}+ C\norm{y^{\epsilon}}_{1}^{2}\norm{\theta}^{2}.
		\end{align*}
		Multiplying the above inequality by  $ 2e^{2\gamma t} $, we get
		\begin{align*}
			\frac{d}{dt}(\norm{\theta}^{2}e^{2\gamma t})&+2\nu e^{2\gamma t} \norm{\theta_{x}}^{2}+ \frac{\nu}{\epsilon}e^{2\gamma t}\theta^{2}(t, 1)+\delta e^{2\gamma t}\norm{\theta}_{L^{4}}^{4}\\[1mm]&\leq Ce^{2\gamma t}\Big(\norm{\eta_{t}}^{2}+C(1+\frac{1}{\epsilon}) \norm{\eta}^{2}+\norm{\eta}_{L^{4}}^{4} + \frac{C}{\epsilon}\eta^{2}(t, 1)\Big)+ (2\alpha +\frac{2\alpha}{M}+2\gamma+ \frac{4\nu r^{2}}{3\epsilon})e^{2\gamma t}\norm{\theta}^{2}\\[1mm]&\quad + Ce^{2\gamma t}\norm{y^{\epsilon}}_{1}^{2}\norm{\theta}^{2}.
		\end{align*}
		Using the Poincaré inequality and  Sobolev embedding, we deduce
		\begin{align*}
			\frac{d}{dt}(\norm{\theta}^{2}e^{2\gamma t})&+(2\nu -\gamma-\alpha -\frac{\alpha}{M} -\frac{2\nu r^{2}}{3\epsilon})e^{2\gamma t} \norm{\theta_{x}}^{2}+ \frac{\nu}{\epsilon}e^{2\gamma t}\theta^{2}(t, 1)+\delta e^{2\gamma t}\norm{\theta}_{L^{4}}^{4}\\[1mm]&\leq Ce^{2\gamma t}(\norm{\eta_{t}}^{2}+C (1+\frac{1}{\epsilon})\norm{\eta}^{2}+\norm{\eta}_{1}^{4} +\frac{C}{\epsilon} \eta^{2}(t, 1)) \ \ +Ce^{2\gamma t}\norm{y^{\epsilon}}_{1}^{2}\norm{\theta}^{2}.
		\end{align*}
		Applying Gronwall's inequality along with Lemma \ref{L3.2}, estimate \eqref{3.3}, and Theorem \ref{th2.21} with $ \gamma=0 $, we get
		\begin{align*}
			e^{2\gamma t}\norm{\theta}^{2}+ \beta_{2} \int_{0}^{t}e^{2\gamma s} \norm{\theta_{x}(s)}^{2} ds &+ \frac{\nu}{\epsilon}\int_{0}^{t}e^{2\gamma s}\theta^{2}(s, 1) ds +\delta \int_{0}^{t}e^{2\gamma s}\norm{\theta(s)}_{L^{4}}^{4}ds\\[1mm]&\leq Ch^{4}\left( \int_{0}^{t}(\norm{y_{t}^{\epsilon}(s)}_{2}^{2}+(1+\frac{1}{\epsilon})\norm{y^{\epsilon}(s)}_{2}^{2} )ds\right)e^{(C\norm{y_{0}}_{1}^{2})},
		\end{align*}
		where \(\beta_{2}=(\nu -\gamma-\alpha -\frac{\alpha}{M})\geq 0\).
		
		Finally, multiplying both sides by  $ e^{-2\gamma t},$ and using Theorem \ref{th2.21} completes the proof.
	\end{proof}
	\begin{remark}
		\label{r3.1}
		Using the Poincar\'e inequality, we have
		\[\norm{\theta}^{2}\leq C\norm{\theta_{x}}^{2}.\]
		Multiplying the above inequality by \(e^{2\gamma t}\) and integrating over the interval \([0, t]\), we obtain from Lemma \ref{L3.3}
		\begin{align}
			\int_{0}^{t}e^{2\gamma s}\norm{\theta(s)}^{2}ds \leq C(\norm{y_{0}}_{3}^{2})h^{4}(1+\frac{1}{\epsilon}).
		\end{align}
	\end{remark}
	The next lemma proves the estimate of $\theta$ in the \(H^{1}\)-norm.
	\begin{lemma}
		\label{L5.4}
		Let assumption $(A1)$ and the regularity result \eqref{r} hold. Then, there exists a positive constant $ C=C(\norm{y_{0}}_{3}, \norm{y_{0}}_{1}) $ independent of \(h\) and \(\epsilon\) such that
		\begin{align*}
			\nu \norm{\theta_{x}}^{2}&+ \frac{\delta}{2}\norm{\theta}_{L^{4}}^{4}+ \frac{\nu}{2\epsilon}\theta^{2}(t, 1))+ e^{-2\gamma t}\int_{0}^{t}e^{2\gamma s}\norm{\theta_{t}(s)}^{2}ds\leq Ce^{-2\gamma t}h^{4}(1+\frac{1}{\epsilon}).
		\end{align*}
	\end{lemma}
	\begin{proof}
		Choose $ \chi =\theta_{t} $ in \eqref{3.5} to have
		\begin{align}
			\label{3.7}
			\nonumber
			\norm{\theta_{t}}^{2}+ \frac{\nu}{2}\frac{d}{dt}\norm{\theta_{x}}^{2}&+ \frac{\nu}{2\epsilon}\frac{d}{dt}(\theta^{2}(t,1))-\alpha (\theta, \theta_{t})+\frac{\delta}{4}\frac{d}{dt}\norm{\theta}_{L^{4}}^{4}\\[1mm]&=((\eta_{t}, \theta_{t})-\alpha(\eta, \theta_{t})) +\frac{\nu}{\epsilon}\eta(t, 1)\theta_{t}(t, 1) +\delta(\eta^{3}, \theta_{t})-3\delta (\eta (y^{\epsilon})^{2}, \theta_{t}) \nonumber\\[1mm]& \ \ +3\delta (\eta^{2} y^{\epsilon}, \theta_{t})-3\delta (\theta (y_{h}^{\epsilon})^{2}, \theta_{t})+3\delta (\theta^{2} y_{h}^{\epsilon}, \theta_{t}) -\frac{\nu}{\epsilon}(u^{\epsilon}(t)-u_{h}^{\epsilon}(t)) \theta_{t}(t, 1),\nonumber\\[1mm]&=:\sum_{i=1}^{5}I_{1}(\theta_{t}).
		\end{align}
		On the right hand side of \eqref{3.7}, the first term $ I_{1}(\theta_{t}) $ yields
		\begin{align*}
			I_{1}(\theta_{t})=\left((\eta_{t}, \theta_{t})- \alpha(\eta, \theta_{t})\right)\leq C(\norm{\eta_{t}}^{2}+ \norm{\eta}^{2})+ \frac{1}{8}\norm{\theta_{t}}^{2}.
		\end{align*}
		In \eqref{3.7}, the second term $ I_{2}(\theta_{t}) $ on the right side can be estimated by
		\begin{align*}
			I_{2}(\theta_{t})=\delta(\eta^{3}, \theta_{t})\leq C\norm{\eta}_{L^{6}}^{6}+\frac{1}{8}\norm{\theta_{t}}^{2}.
		\end{align*}
		Substituting $ y_{h}^{\epsilon}=\theta+ \tilde{y}_{h}^{\epsilon} $ into the third term $ I_{3}(\theta_{t}) $ on the right hand side of \eqref{3.7}, we arrive at
		\begin{align*}
			I_{3}(\theta_{t})=3\delta (\eta (y^{\epsilon})^{2}, \theta_{t})-3\delta (\eta^{2} y^{\epsilon}, \theta_{t})-3\delta(\theta^{2}\tilde{y}_{h}^{\epsilon}, \theta_{t})-3\delta(\theta (\tilde{y}_{h}^{\epsilon})^{2}, \theta_{t}).
		\end{align*}
		Using the Cauchy-Schwarz inequality and Young's inequality to the above equation, we observe that
		\begin{align*}
			I_{3}(\theta_{t})\leq C\norm{\eta}^{2}\norm{y^{\epsilon}}_{\infty}^{4}+ C\norm{\eta}_{L^{4}}^{4}\norm{y^{\epsilon}}_{\infty}^{2}+C\norm{\theta}^{2}\norm{\tilde{y}_{h}^{\epsilon}}_{\infty}^{4}+ C\norm{\theta}_{L^{4}}^{4}\norm{\tilde{y}_{h}^{\epsilon}}_{\infty}^{2}+\frac{1}{8}\norm{\theta_{t}}^{2}.
		\end{align*}
		The fourth term $ I_{4}(\theta_{t}) $ on the right hand side of \eqref{3.7} is bounded by
		\begin{align*}
			I_{4}(\theta_{t})=\frac{\nu}{\epsilon}\eta(t, 1)\theta_{t}(t, 1)
			\leq \frac{\nu}{\epsilon}\frac{d}{dt}(\eta(t, 1)\theta(t, 1))+C\eta_{t}^{2}(t, 1)+\frac{4 r^{2}}{3\epsilon^{2}}\theta^{2}(t, 1).	
		\end{align*}
		Lastly, the final term $ I_{5}(\theta_{t})$  on the right hand side of \eqref{3.7} is estimated by
		\begin{align*}
			I_{5}(\theta_{t})=\frac{\nu}{\epsilon}(u^{\epsilon}(t)-u_{h}^{\epsilon}(t)) \theta_{t}(t, 1)\leq \frac{\nu}{\epsilon}\frac{d}{dt}((u^{\epsilon}(t)-u_{h}^{\epsilon}(t)) \theta(t, 1))+\frac{1}{8}\norm{\theta_{t}}^{2}+\frac{4r^{2}}{3\epsilon^{2}}\theta^{2}(t, 1)+C\norm{\eta_{t}}^{2}.
		\end{align*}
		Substituting all these estimates into \eqref{3.7} and multiplying by $ 2 e^{2\gamma t} $ in the resulting inequality, we get
		\begin{align*}
			\frac{d}{dt}(e^{2\gamma t}(\nu \norm{\theta_{x}}^{2}+ \frac{\delta}{2}\norm{\theta}_{L^{4}}^{4}&+ \frac{\nu}{\epsilon}\theta^{2}(t, 1)))+ e^{2\gamma t}\norm{\theta_{t}}^{2}\\[1mm]&\leq \frac{d}{dt}\left(e^{2\gamma t}\Big(\frac{\alpha}{2}\norm{\theta}^{2} + \frac{\nu}{\epsilon}\eta(t, 1)\theta(t,1) +\frac{\nu}{\epsilon}(u^{\epsilon}(t)-u_{h}^{\epsilon}(t))\theta(t, 1) \Big)\right)\\[1mm]& \ \ + 2\gamma e^{2\gamma t}\left(\nu \norm{\theta_{x}}^{2}+ \frac{\delta}{2}\norm{\theta}_{L^{4}}^{4}+ \frac{\nu}{\epsilon}\theta^{2}(t, 1)\right)\\[1mm]& \ \ -2\gamma e^{2\gamma t}\left(\frac{\alpha}{2}\norm{\theta}^{2} + \frac{\nu}{\epsilon}\eta(t, 1)\theta(t,1) +\frac{\nu}{\epsilon}(u^{\epsilon}(t)-u_{h}^{\epsilon}(t))\theta(t, 1) \right)\\[1mm]& \ \ +Ce^{2\gamma t}\left(\norm{\eta_{t}}^{2}+ \norm{\eta}^{2}+\norm{\eta}_{L^{6}}^{6}+ \eta_{t}^{2}(t, 1)+ \norm{\eta}_{L^{4}}^{4}\norm{y^{\epsilon}}_{\infty}^{2}+\norm{\eta}^{2}\norm{y^{\epsilon}}_{\infty}^{4}\right)\\[1mm]& \ \ +Ce^{2\gamma t}\norm{\tilde{y}_{h}^{\epsilon}}_{\infty}^{2}(\norm{\theta}_{L^{4}}^{4}+\norm{\theta}^{2}\norm{\tilde{y}_{h}^{\epsilon}}_{\infty}^{2} )+ \frac{16r^{2}}{3\epsilon^{2}}e^{2\gamma t}\theta^{2}(t, 1).
		\end{align*}
		Integrating with respect to time  over \([0,t]\)  and using  Young's inequality, we obtain 
		\begin{align*}
			e^{2\gamma t}(\nu \norm{\theta_{x}}^{2}&+ \frac{\delta}{2}\norm{\theta}_{L^{4}}^{4}+ \frac{\nu}{2\epsilon}\theta^{2}(t, 1))+ \int_{0}^{t}e^{2\gamma s}\norm{\theta_{t}}^{2}ds\\[1mm]&\leq Ce^{2\gamma t}\norm{\theta}^{2}+ C\max\{1, \frac{1}{\epsilon}\}\int_{0}^{t}e^{2\gamma s}\left(\norm{\theta_{x}}^{2}+ \norm{\theta}_{L^{4}}^{4}+ \theta^{2}(s, 1)\right)ds+ C(\eta^{2}(t, 1)+\norm{\eta}^{2})\\[1mm]& \ \ +C\int_{0}^{t}e^{2\gamma s}\left(\norm{\eta_{t}}^{2}+\frac{1}{\epsilon} \norm{\eta}^{2}+\norm{\eta}_{L^{6}}^{6}+ \eta_{t}^{2}(s, 1)+\frac{1}{\epsilon}\eta^{2}(s,1)+ \norm{\eta}_{L^{4}}^{4}\norm{y^{\epsilon}}_{\infty}^{2}+\norm{\eta}^{2}\norm{y^{\epsilon}}_{\infty}^{4}\right)ds\\[1mm]& \ \ +C\max\{1, \frac{1}{\epsilon}\}\int_{0}^{t}e^{2\gamma s}\left(\norm{\tilde{y}_{h}^{\epsilon}}_{\infty}^{2}(\norm{\theta}_{L^{4}}^{4}+\norm{\theta}^{2}\norm{\tilde{y}_{h}^{\epsilon}}_{\infty}^{2} )+ \frac{1}{\epsilon}\theta^{2}(s, 1) \right) ds.
		\end{align*}
		Using \eqref{3.3} along with Lemmas \ref{L3.2} – \ref{L3.3} and Theorem \ref{th2.21}, and noting that $\norm{\tilde{y}_{h}^{\epsilon}}_{\infty} \leq C\norm{y^{\epsilon}}_{1}$, we apply the Sobolev embedding and Young's inequality to the above estimate. Consequently, we have	
		\begin{align*}
			e^{2\gamma t}(\nu \norm{\theta_{x}}^{2}&+ \frac{\delta}{2}\norm{\theta}_{L^{4}}^{4}+ \frac{\nu}{2\epsilon}\theta^{2}(t, 1))+ \int_{0}^{t}e^{2\gamma s}\norm{\theta_{t}}^{2}ds\leq C(\norm{y_{0}}_{3}^{2})h^{4}(1+\frac{1}{\epsilon}).
		\end{align*}
		The proof is completed upon multiplying both sides by $ e^{-2\gamma t}$.
	\end{proof}
	\begin{remark}
		\label{r5.21}
		From Lemma \ref{L3.3}, we arrive at 
		\[ \|y^{\epsilon}-y^{\epsilon}_{h}\|\leq C(h^{2}+\frac{h^{2}}{\sqrt{\epsilon}}),\] and 
		\[\|u^{\epsilon}-u^{\epsilon}_{h}\|=r\|\int_{0}^{1}x(\eta - \theta)dx\|\leq C r(h^{2}+\frac{h^{2}}{\sqrt{\epsilon}}),\]
		where \(u^{\epsilon}=u^{\epsilon}(t)\).
	\end{remark}
	Finally, we state the following theorem, which establishes the error estimates between the closed loop system \eqref{ch1}-\eqref{c} and the penalized feedback control problem \eqref{p1}-\eqref{p2}, for both the state variable and the control input.
	\begin{theorem}
		\label{th3.2}
		Suppose that hypothesis of Lemma \ref{L3.3} and Theorem \ref{th2.1} holds, then we get
		\begin{align*}
			\norm{y-y_{h}^{\epsilon}}_{i}\leq C(\norm{y_{0}}_{3})\Big(\sqrt{\epsilon}+h^{2-i}+\frac{h^{2-i}}{\sqrt{\epsilon}}\Big), \quad i=0, 1,
		\end{align*}
		and
		\[
		\norm{u-u_{h}^{\epsilon}}\leq C(\norm{y_{0}}_{3}) r \Big(\sqrt{\epsilon}+ h^{2}+\frac{h^{2}}{\sqrt{\epsilon}}\Big),
		\]
        where \(\norm{.}_{0}=\|.\|.\)
	\end{theorem}
	\begin{proof}
		Since
		\begin{align*}
			y-y_{h}^{\epsilon}=y-y^{\epsilon}+y^{\epsilon}-y_{h}^{\epsilon}=y-y^{\epsilon}+(\eta- \theta),
		\end{align*}
		from Lemmas \ref{L3.3} -\ref{L5.4} and Theorem \ref{th2.1} with the triangle inequality, we obtain
		\begin{align*}
			\norm{y-y_{h}^{\epsilon}}_{i}\leq C\Big(\sqrt{\epsilon}+h^{2-i}+\frac{h^{2-i}}{\sqrt{\epsilon}}\Big), \quad i=0, 1.
		\end{align*}
		The first part of the proof is completed.
		
		Since $u(t)=-r\int_{0}^{1}xy(t,x)dx, \quad  u^{\epsilon}(t)=-r\int_{0}^{1}xy^{\epsilon}(t,x)dx, $ and $ u_{h}^{\epsilon}(t)=-r\int_{0}^{1}xy_{h}^{\epsilon}(t,x)dx,$ we can write
		\begin{align*}
			u-u_{h}^{\epsilon}=u-u^{\epsilon} -r\int_{0}^{1}x(\eta - \theta)dx=r\int_{0}^{1}xz(t,x)dx-r\int_{0}^{1}x(\eta - \theta)dx.
		\end{align*}
		Hence, using Theorem \ref{th2.1} and Lemma \ref{L3.3}  with $ \gamma=0 $, we observe that
		\begin{align*}
			\norm{u-u_{h}^{\epsilon}}^{2}\leq C(\norm{y_{0}}_{3}^{2}) r^{2}\Big(\epsilon+ h^{4}+\frac{h^{4}}{\epsilon}\Big).
		\end{align*}
		This completes the second part of the proof.
	\end{proof}
	%In the following theorem, we determine the error estimate of the penalized feedback control input.
	%\begin{theorem}
	%	\label{th3.1}
	%	Let $ y_{0}\in H^{3}(0,1). $ Then the following holds
	%	\begin{align*}
		%	\norm{u(t)-u_{h}^{\epsilon}(t)}\leq C r \Big(\sqrt{\epsilon}+ h^{2}+\frac{h^{2}}{\sqrt{\epsilon}}\Big)\norm{y_{0}}_{3}.
		%	\end{align*}
	%\end{theorem}
	%\begin{proof} Since $u(t)=-r\int_{0}^{1}xy(t,x)dx, \quad  u^{\epsilon}(t)=-r\int_{0}^{1}xy^{\epsilon}(t,x)dx, $ and $ u_{h}^{\epsilon}(t)=-r\int_{0}^{1}xy_{h}^{\epsilon}(t,x)dx,$ therefore, we can write
	%	\begin{align*}
		%	u(t)-u_{h}^{\epsilon}(t)=u(t)-u^{\epsilon}(t) -r\int_{0}^{1}x(\eta - \theta)dx=r\int_{0}^{1}xz(t,x)dx-r\int_{0}^{1}x(\eta - \theta)dx.
		%	\end{align*}
	%	Hence, using Theorem \ref{th2.1} and Lemma \ref{L3.3}  with $ \gamma=0 $, we observe that
	%	\begin{align*}
		%	\norm{u(t)-u_{h}^{\epsilon}(t)}^{2}\leq C r^{2}\Big(\epsilon+ h^{4}+\frac{h^{4}}{\epsilon}\Big)\norm{y_{0}}_{3}^{2}.
		%	\end{align*}
	%\end{proof}
	\begin{remark}
		\label{r5.2}
		If we choose \(\epsilon= h^{l}\), \( \  l>0\), then we arrive at from Theorem \ref{th3.2} 
		\[\norm{y-y_{h}^{\epsilon}}\leq C h^{\gamma_{1}},\]
		and 
		\[\norm{u-u_{h}^{\epsilon}}\leq C r h^{\gamma_{1}},\]
		where \(\gamma_{1}=\min \{\frac{l}{2}, 2, 2-\frac{l}{2}\} \) and \(l<4\) . 
	\end{remark}
	
	\section{Numerical Examples.}
	In this section, we present numerical experiments to illustrate the stabilization of the Chafee-Infante equation and to evaluate the errors in both the state variable and the control input using a penalization approach for the Dirichlet boundary control problem. In addition, we verify the convergence of the solution of the penalized feedback control problem \eqref{p1}-\eqref{p2} to the solution of the original closed-loop system \eqref{ch1}-\eqref{c}, as the penalty parameter \(\epsilon\rightarrow 0\). Two numerical strategies are employed: (i) direct computation of the original Dirichlet closed-loop system, and (ii) successive approximation of the penalized problem following \cite{ravindran2017finite} and \cite{18M1233716}. The computational result of the original closed loop system \eqref{ch1}-\eqref{c} with Dirichlet boundary feedback control \cite{19M1252235} is shown using the transformation technique discussed in \cite{kang1991unbounded}.
	
	We now describe a fully discrete version of the penalized finite element scheme \eqref{3.1}, using the backward Euler method for time discretization. Let $0 < k < 1$ denote the time step size, and define the time levels by $t_{n} = nk$, where $n$ is a nonnegative integer. For any sufficiently smooth function $\phi$ defined on $[0, T]$, we set $\phi^{n} = \phi(t_{n})$ and introduce the backward difference operator by $\bar{\partial}_{t} \phi^{n} = \frac{\phi^{n} - \phi^{n-1}}{k}$. We denote  \(Y^{n}=Y_{h}^{n}\in S_{h}\) as the approximation of \(y^{\epsilon}(t)\) at \(t=t_{n}\).
	
	The backward Euler method is applied to the semi-discrete scheme \eqref{3.1} to find a sequence $\{Y^{n}\}_{n \geq 1}$ such that
	\begin{align}
		\label{4.1}
		(\bar{\partial}_{t}Y^{n}, \chi) + \nu (Y_{x}^{n}, \chi_{x}) + \frac{\nu}{\epsilon}Y^{n}(1)\chi(1) + \delta\Big((Y^{n})^{3}, \chi\Big) = \frac{\nu}{\epsilon}u_{h}^{\epsilon}(t_{n})\chi(1) + \alpha(Y^{n}, \chi), \quad \forall \ \chi \in S_{h},
	\end{align}
	 where the control input is defined as
	\begin{align}
		\label{6.2}
			u_{h}^{\epsilon}(t_{n}) = -r\int_{0}^{1}x Y^{n}(x) \, dx.
	\end{align}
    The initial approximation is taken as $Y^{0} = y_{0h}$.
	
	We solve the nonlinear system \eqref{4.1} using Newton’s method  to compute $Y^{n+1}$, with the iteration initialized by the previous time level solution $Y^{n}$.
	For notational convenience, we denote the discrete solution of the penalized system by  \( Y^{\epsilon} \), and that of the original closed-loop system \eqref{ch1}-\eqref{c} by \(Y\). Since the exact solution of the penalized system is not known, we take the refined mesh solution (reference solution) as the exact solution for error evaluation. Here,
    We denote this reference state by $y^{\epsilon}_{r}$ and the corresponding control input by $u^{\epsilon}_{r}$.
	
	In the following example, our aim is to show the asymptotic behavior of the penalized feedback control problem under varying choices of the penalty parameter $\epsilon$. Moreover, we study the order of convergence for both the state variable and control input. In addition, we demonstrate the convergence $Y^{\epsilon}$ to $Y$ as \(\epsilon \) goes to zero.
	\begin{example}
		\label{ex1}
		We consider the initial condition $\ y_{0}(x) = \sin(\pi x)$, at $t=0$ for $\ x\in
		(0, 1)$. The diffusion coefficient is set to $\nu = 0.2$, and the control gain is taken as $r = \sqrt{\epsilon}$. We choose the space step size \(h=\frac{1}{100}\) and the time step size \(k=\frac{1}{900}\). We set the time interval $[0,1]$.
	\end{example}
	
	The Chafee-Infante equation \eqref{ch1}-\eqref{ch2} with zero Dirichlet boundary condition is asymptotically stable for $\alpha = 0.19 = \delta$ since \(\alpha\leq 2\nu\); this corresponds to the ``Zero Dirichlet solution'' in Figure \ref{fig:ex1}(i) with $\epsilon=0$ and $r = 0$, see Remark \ref{r2.2}. However, when the feedback control input $u(t)$ and the penalized feedback control input $u^{\epsilon}(t)$ are applied, the solution decays faster than in the zero Dirichlet boundary case. These are denoted in the same figure by ``Controlled Dirichlet solution" and ``Controlled penalized solution”, respectively.
    
	From Figures \ref{fig:ex1}(ii) and (iii), we observe that the state variable and the control input in the $L^{2}$-norm converge to the solution of the Dirichlet problem for different values of $\epsilon$. Moreover, for various values of \(\alpha\)  satisfying Assumption $(A1)$, the trajectories of both the state variable and the control input remain unchanged, and eventually all of them decay to zero as $t\to \infty$.
	\begin{figure}[ht]
		\centering
		(i)\includegraphics[width= 0.45\textwidth]{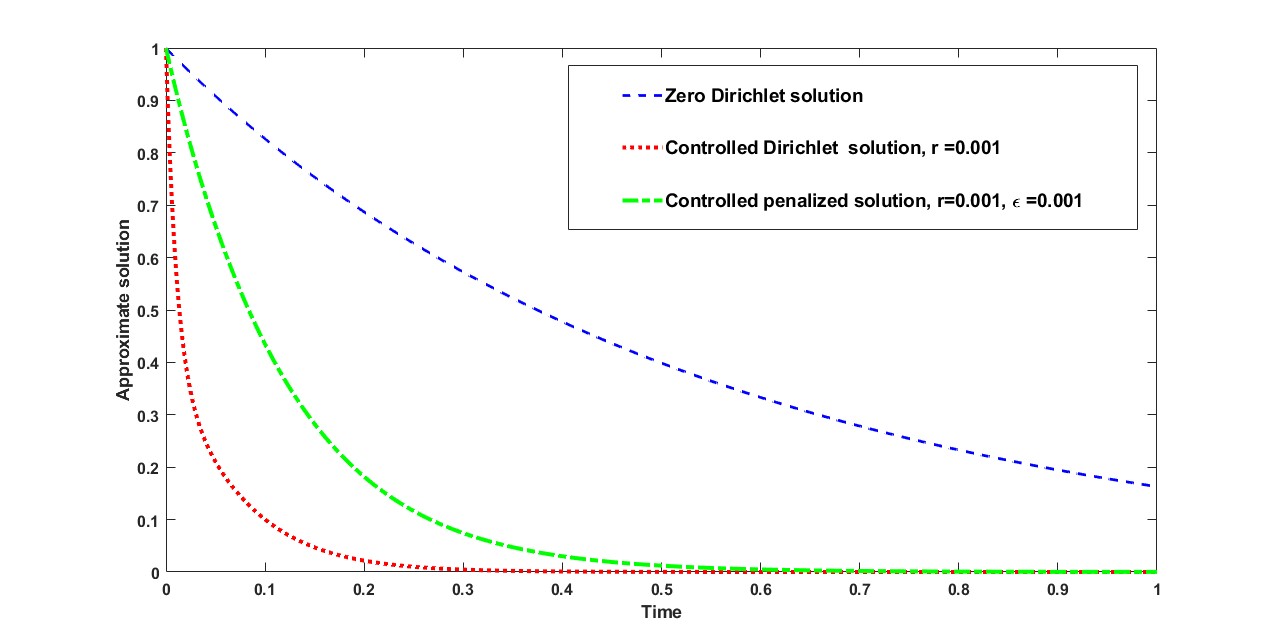}
		(ii)\includegraphics[width= 0.46\textwidth]{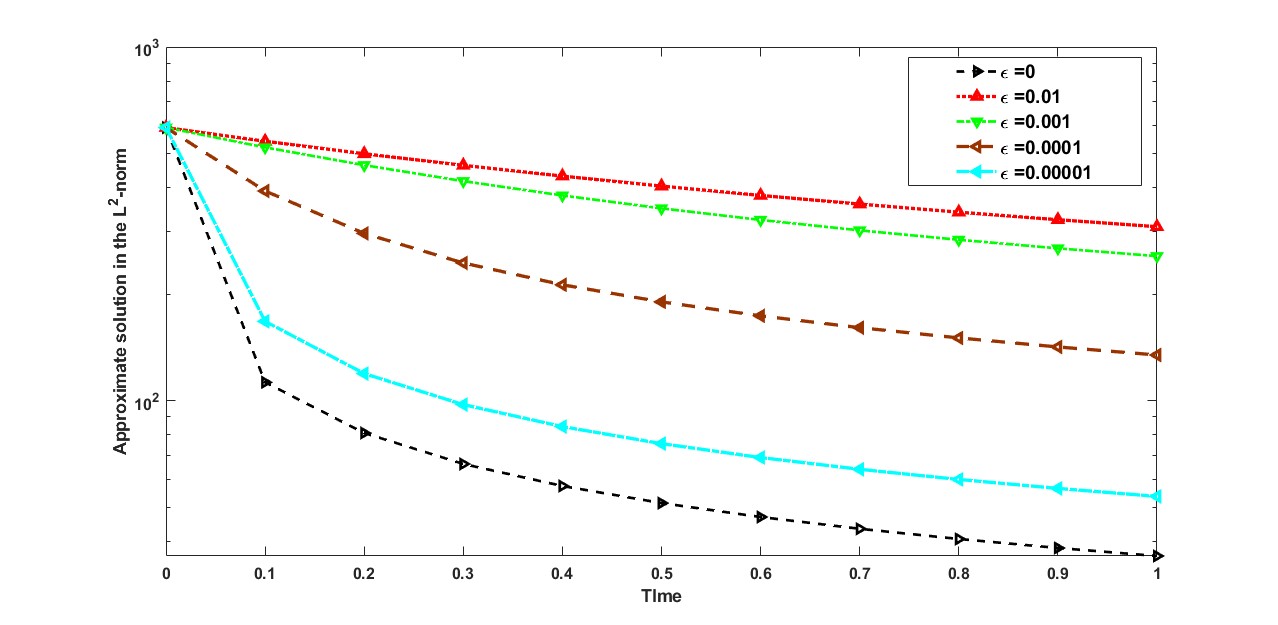}
		(iii) \includegraphics[width= 0.45\textwidth]{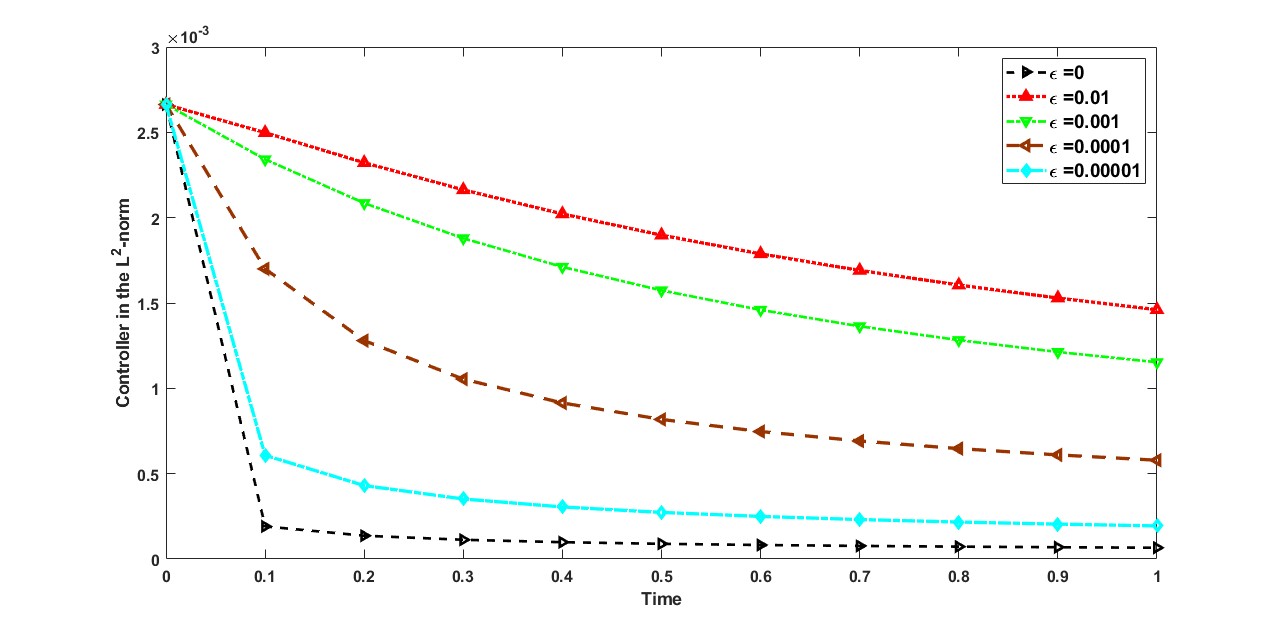}
		(iv) \includegraphics[width= 0.45\textwidth]{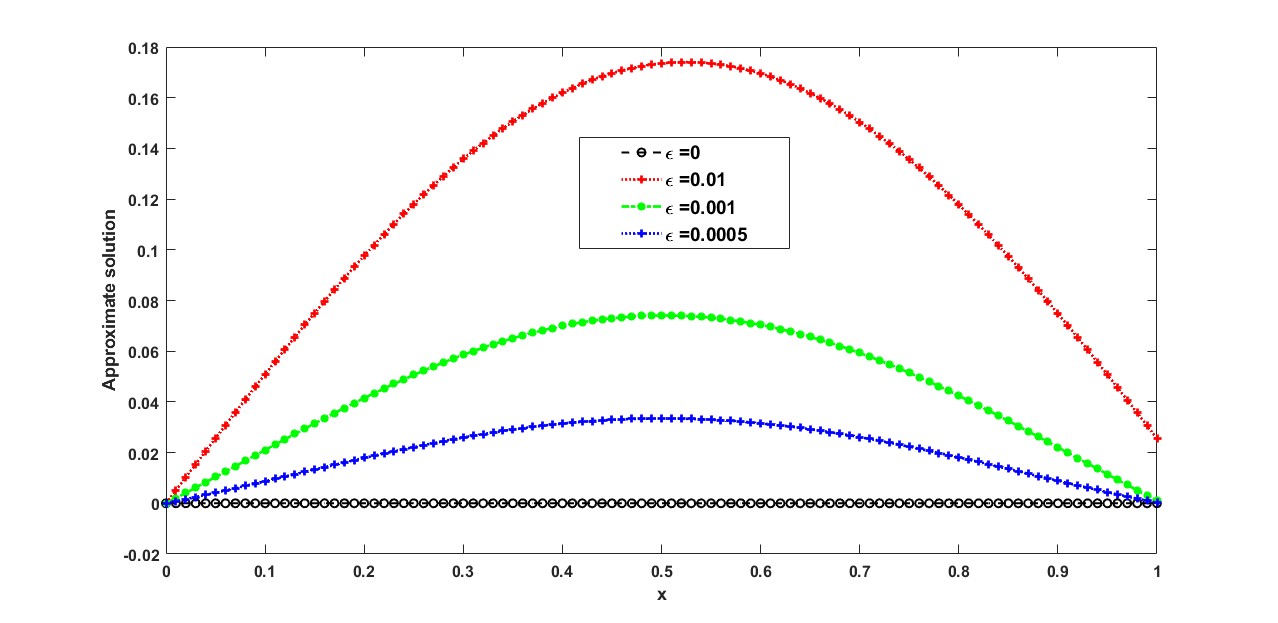}
		\caption{Example \ref{ex1}: \textbf{(i)} Approximate solution at $x=0.5$. \textbf{(ii)} Approximate solution in the $L^{2}$-norm for various values of $\epsilon$. \textbf{(iii)} Controller in the $L^{2}$-norm for different values of $\epsilon$. \textbf{(iv)} Approximate solution in the $L^{2}$-norm.}
		\label{fig:ex1}	
	\end{figure}	
    
To further validate the accuracy of the penalization approach, we now show that the solution $Y^{\epsilon}$ of the penalized system converges to the solution $Y$ of the original system as \(\epsilon \rightarrow 0\). Since the exact solution of the closed loop system \eqref{ch1}-\eqref{c} is not known, we compute it numerically using a finite element method, where the Dirichlet boundary conditions are imposed exactly at the boundary mesh points.  Moreover, for the convergence analysis, we choose \(r\) sufficiently small. In Figures \ref{fig:ex1}(ii)-(iv), \(\epsilon=0\) represents the solution of the closed loop system \eqref{ch1}-\eqref{c}, together with the penalized solutions of the state variable and control input for various values of \(\epsilon\) with $k=\frac{1}{7050}$. We observe that the solutions of the penalized control problem converge to the solution of the original closed-loop system \eqref{ch1}-\eqref{c} as $\epsilon\rightarrow 0$. This convergence is supported theoretically by the analysis in Subsection \ref{sbs4.1}. In the second approach, we study the convergence behavior of the penalized problem. From Table \ref{table:4}, we observe that the error between two successive approximations of the state variable $Y^{\epsilon}$ and the control input \(u^{\epsilon}\) decreases for different values of the penalty parameter $\epsilon$ in both the $L^{2}$ and $L^{\infty}$-norms.

\begin{table}[ht!]
		\centering
		\caption{ Error in $ L^{2}$ and $L^{\infty}-$norms  in successive approximate solution $(E_{\epsilon_{i}}:=\|Y^{\epsilon_{i}}-Y^{\epsilon_{i-1}}\|)$ of the state variables and $(E_{\epsilon_{i}}:=\|u^{\epsilon_{i}}-u^{\epsilon_{i-1}}\|)$ of the control variables for different values of $ \epsilon $ with \(r=\sqrt{\epsilon}\) in Example \ref{ex1}. }
		\begin{tabular}{ c|| c|| c|| c ||c|| c|| c|| }
			\toprule
			$\epsilon$&  $\|Y^{\epsilon}\|$ &  $\|Y^{\epsilon}\|_{\infty}$  & 	$ \|E_{\epsilon_{i}}\|$ & $ \|E_{\epsilon_{i}}\|_{\infty}$ &  $\|u^{\epsilon}\|_{\infty}$       & $ \|E_{\epsilon_{i}}\|_{\infty}$    \\
			\midrule
			$10^{0}$ & $0.0019$   &   $3.33e-04$ &   $ -- $ &     $-- $         &    $0.3183$      &  $--$       \\ 
			$10^{-1}$ & $0.0012$   &  $2.22e-04$ &  $ 0.0022 $&    $ 3.70e-04 $  &  $ 0.1007$       &  $0.2177$    \\
			$10^{-2}$ &  $ 6.92e-05$ &   $1.30e-05$ & $ 0.0011$ &   $ 2.09e-04 $ &  $ 0.0318$   &  $0.0688$     \\
			$10^{-3}$ &  $ 6.38e-06$ &   $1.18e-06$& $6.28e-05 $ &  $1.18e-05$     &  $ 0.0101$  &  $0.0218 $     \\
			$10^{-4}$ &  $1.09e-06$ &   $2.18e-07$ & $ 5.53e^-06$ &  $ 1.0435e-06$ &  $0.0032$  & $ 0.0069$   \\
			$10^{-5}$ &  $7.71e-07$ &  $1.54e-07$ & $ 5.49e-07 $ &   $1.03e-07 $   &  $0.0010$   &  $0.0022$  \\
			$10^{-6}$&   $7.48e-07$ &  $1.47e-07$ & $ 5.66e-08 $ &   $ 1.05e-08 $  &  $3.18e-04$ & $6.88e-04  $ \\
			\bottomrule
		\end{tabular}
		\label{table:4}
	\end{table}
    
	Next, to numerically verify the spatial order of convergence of the penalized feedback control problem, we choose \(\epsilon = c h^{l}\), where \(c\) is a small positive constant (see \cite{dione2015penalty} for details). In the remaining part of this example, we use the refined mesh solution with $h=\frac{1}{2048}$ as the reference solution.
   We choose the time step size $k=\frac{1}{1050}$.
	% From Table \ref{table:1}, for $\epsilon=0.01h^{2}$, we observe that the order of convergence for the state variable is two in both the $L^{2}$ and $L^{\infty}$-norms and the error decays for various values of $h$ with a fixed time step size $k$. However, Remark ~\ref{r5.21}, theoretically predicts only first-order convergence for  \(l = 2\). 
    Table \ref{table:2} shows that the order of convergence of the state variable is approximately $1.39$ for $\epsilon=0.01h^{\frac{4}{3}}$ in both the $L^{2}$ and $L^{\infty}$-norms, which is consistent with the result stated in Remark~\ref{r5.21} for \(l = \frac{4}{3}\). For $\epsilon=0.01h^{2}$, we observe that the order of convergence for the state variable is two in both the $L^{2}$ and $L^{\infty}$-norms. However, Remark ~\ref{r5.21}, theoretically predicts only first-order convergence for  \(l = 2\).
    
	  Table \ref{table:3} shows that the order of convergence of the control input for $ \epsilon = 0.01h^{2}$ is one in the $ L^{\infty} $-norm, which supports the theoretical result in Remark~\ref{r5.21}. However, for $\epsilon=0.01 h^{4/3}$, the observed order of convergence is not consistent with the theoretical prediction in Remark~\ref{r5.21} for the control input.
	
	% \begin{table}[H]
	% 	\centering
	% 	\caption{ The order of convergence (O.C.) for the state variable with respect to space in Example \ref{ex1}. }
	% 	\begin{tabular}{ c|| c|| c|| c|| c }
	% 		\toprule
	% 		h &  $\|y^{\epsilon}_{r}-Y^{\epsilon}\|$ & O. C. & $\|y^{\epsilon}_{r}-Y^{\epsilon}\|_{\infty}$ &          O. C.    \\
	% 		\midrule
	% 		$\frac{1}{8}$ &   $ 8.0318e-09$       &     $--  $     &  $1.4603e-08 $ &    $--$                 \\
			
	% 		$\frac{1}{16}$&   $  2.1589e-09 $     &  $ 1.89 $     &   $4.0466e-09$ &    $1.85$           \\
			
	% 		$\frac{1}{32}$ &       $5.2416e-10 $      & $2.04$    &  $1.006e-09$&     $2.01$              \\
			
	% 		$\frac{1}{64}$ &      $ 1.266e-10 $     &  $2.04$    &  $2.4389e-10$  &  $2.03$           \\
			
	% 		$\frac{1}{128}$ &        $3.0836e-11 $    & $ 2.03 $   &  $5.9663e-11$  &      $2.03$            \\
			
	% 		$\frac{1}{256}$ &        $ 7.5034e-12$    & $2.04$     &  $1.4546e-11$  &      $2.04$            \\
	% 		\bottomrule
	% 	\end{tabular}
	% 	\label{table:1}
	% \end{table}
	
	\begin{table}[ht!]
		\centering
		\caption{ The order of convergence (O.C.) of the state variable with respect to space in Example \ref{ex1} with the varying values of $h$ and a fixed value of $ k.$ }
		\begin{tabular}{c|| c|| c|| c|| c}
			\toprule
			h &  $\|y^{\epsilon}_{r}-Y^{\epsilon}\|$  & O. C. &  $\|y^{\epsilon}_{r}-Y^{\epsilon}\|_{\infty}$  &          O. C.    \\
			\midrule
			$\frac{1}{8}$ &   $ 3.3996e-07$       &     $--  $     &  $5.6108e-07 $ &    $--$                 \\
			
			$\frac{1}{16}$&   $  1.3826e-07 $     &  $1.29 $     &   $2.3321e-07$ &    $1.27$           \\
			
			$\frac{1}{32}$ &       $5.2825e-08 $      & $1.38$    &  $8.9011e-08$&     $1.38$              \\
			
			$\frac{1}{64}$ &      $ 2.0129e-08 $     &  $1.39$    &  $3.3934e-08$  &  $1.39$           \\
			
			$\frac{1}{128}$ &        $7.6789e-09$    & $ 1.39 $   &  $1.2944e-08$  &      $1.39$            \\
			
			$\frac{1}{256}$ &        $ 2.8825e-09$    & $1.41$     &  $4.8589e-09$  &      $1.41$            \\
			\bottomrule
		\end{tabular}
		\label{table:2}
	\end{table}
	
	\begin{table}[ht]
		\centering
		\caption{ The order of convergence (O. C.) of control input with respect to space in Example \ref{ex1} with the varying values of $h$ and a fixed value of $ k. $ }
		\begin{tabular}{c||c|| c || c ||c|| }
			\toprule
			h & $\|u^{\epsilon}_{r}-u^{\epsilon}_{h}\|_{\infty}$, $ \epsilon=0.01h^2 $        & O. C. & $\|u^{\epsilon}_{r}-u^{\epsilon}_{h}\|_{\infty}$, $ \epsilon=0.01h^{\frac{4}{3}}$   &  O. C.    \\
			\midrule
			$\frac{1}{8}$&   $  3.9121e-04 $ &  $--$ &   $0.0077$ & $--$           \\
			$\frac{1}{16}$&   $  1.9675e-04 $ &  $0.99$ &   $0.0048$ & $0.68$           \\
			
			$\frac{1}{32}$ & $9.7838e-05 $  & $1.00$  &  $0.0030$& $0.69$              \\
			
			$\frac{1}{64}$ & $4.8172e-05 $  & $1.02$  &  $0.0018$ &  $0.72$           \\
			
			$\frac{1}{128}$ & $2.3312e-05 $ & $ 1.04$ &  $0.0011$ & $0.76$            \\
			
			$\frac{1}{256}$ & $ 1.0880e-05$    & $1.09$  &  $5.9212e-04$  &      $0.83$ \\          
			\bottomrule
		\end{tabular}
		\label{table:3}
	\end{table}

    We now verify the  order of convergence between the original closed-loop system \eqref{ch1}–\eqref{c} and the penalized feedback control problem \eqref{p1}-\eqref{p2}. First, we solve the original closed loop system using $r=0.0001$ and $k= \frac{1}{7050}.$ The penalized system is then solved for varying values of the spatial step size $h$, while keeping $k$ fixed. From Table \ref{table:5}, we observe that for \(\epsilon=0.01h\), the order of convergence of the state variable is one in both the \(L^{2}\) and \(L^{\infty}\)-norms. Furthermore, the observed order of convergence is approximately two for \(\epsilon= h^{2}\)  and \(1.34\) for \(\epsilon=0.1h^{\frac{4}{3}}\). For the state variable, the order of convergence is not consistent with the theoretical results in Remark \ref{r5.2}.
    
     From Table \ref{table:7}, we notice that the order of convergence of the control input for \(\epsilon=0.01h\) is $0.5$ and for \(\epsilon= h^{2}\) is one in the \(L^{\infty}\)-norm. Furthermore, for \(\epsilon=0.1h^{\frac{4}{3}}\), we observe that the order of convergence of the control input is \(0.68\). Consequently, the order of convergence is consistent with the results of Remark \ref{r5.2} for the control input.

	\begin{table}[H]
		\centering
		\caption{ The order of convergence (O.C.) of the state variable with respect to space in Example \ref{ex1} with the varying values of $h$ and a fixed value of $ k. $ }
		     \begin{tabular}{ c|| c|| c|| c|| c }
			\toprule
			h &  $\|Y-Y^{\epsilon}\|$ & O. C. & $\|Y-Y^{\epsilon}\|_{\infty}$ &          O. C.    \\
			\midrule
			$\frac{1}{8}$ &   $ 4.0204e-06$       &     $--  $     &  $6.49e-06 $ &    $--$                 \\
			
			$\frac{1}{16}$&   $  1.8470e-06$     &  $ 1.22$     &   $3.1112e-06$ &    $1.05$           \\
			
			$\frac{1}{32}$ &       $8.5255e-07 $      & $1.12$    &  $1.4358e-06$&     $1.11$              \\
			
			$\frac{1}{64}$ &      $ 4.0421e-07 $     &  $1.07$    &  $6.8319e-07$  &  $1.07$           \\
			
			$\frac{1}{128}$ &        $1.9594e-07$    & $ 1.04 $   &  $3.3165e-07$  &      $1.04$            \\
			
			$\frac{1}{256}$ &        $ 9.6227e-08$    & $1.02$     &  $1.6318e-07$  &      $1.02$            \\
			\bottomrule
		\end{tabular}
		\label{table:5}
	\end{table}

	 \begin{table}[ht]
		\centering
		\caption{ The order of convergence (O.C.) of control input with respect to space in Example \ref{ex1} with the varying values of $h$ and a fixed value of $ k. $ }
		\begin{tabular}{c||c|| c || c ||c|| }
			\toprule
			h & $\|u-u^{\epsilon}_{h}\|_{\infty}$, $ \epsilon=0.01h $        & O. C. & $\|u-u^{\epsilon}_{h}\|_{\infty}$, $ \epsilon=h^{2}$   &  O. C.    \\
			\midrule
			$\frac{1}{8}$&   $  0.0111 $ &  $--$ &   $0.0392$ & $--$      \\
			$\frac{1}{16}$&   $  0.0079 $ &  $0.49$ &   $0.0198$ & $0.98$    \\
			
			$\frac{1}{32}$ & $0.0056$  & $0.49$  &  $0.0099$& $0.99$   \\
			
			$\frac{1}{64}$ & $0.0039 $  & $0.50$  &  $0.0049$ &  $1.00$   \\
			
			$\frac{1}{128}$ & $0.0028$ & $ 0.50$ &  $0.0025$ & $1.01$    \\
			
			$\frac{1}{256}$ & $ 0.0020$    & $0.51$  &  $0.0012$  &    $1.01$ \\          
			\bottomrule
		\end{tabular}
		\label{table:7}
	\end{table}

	In the next example, we show the behavior of the state variable and the control input in the penalized control problem for varying values of the diffusion coefficient $\nu$ and the cubic nonlinearity coefficient $\delta$. 
	
	\begin{example}
		\label{ex2}
		We choose the initial condition $ y_{0}(x)=x(1-x), \ x\in
		(0, 1).$ We select the parameters $ \alpha =0.1, \delta=0.1, $ $ r=0.001, $ and $ \epsilon=0.001.$ We take the space step size $h=\frac{1}{80}$, and the time step size $k=\frac{1}{900}.$ 
	\end{example}	
	Figures \ref{fig:ex2}(i) and (ii) display the state variable and control input in the $L^{2}$-norm for different values of \(\nu,\) respectively. For large values of \(\nu\), both the control input and the state variable lead to a faster decay toward zero as \(t\rightarrow \infty\).  Figures \ref{fig:ex2}(iv) and (v) present both the state variable and the control input for varying \(\delta\) with \(\nu=0.1\).

	\begin{figure}[ht]
		\centering
		(i)\includegraphics[width= 0.45\textwidth]{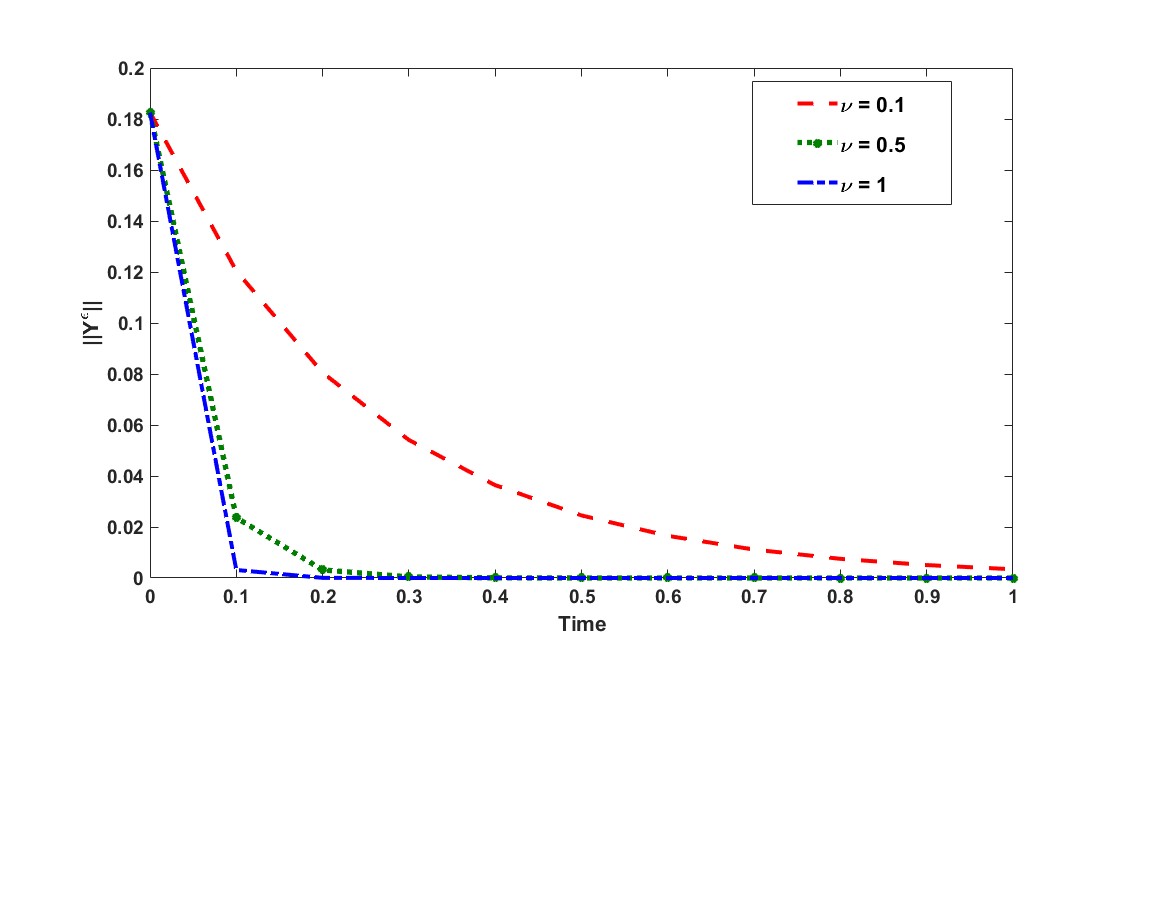}
		(ii) \includegraphics[width= 0.44\textwidth]{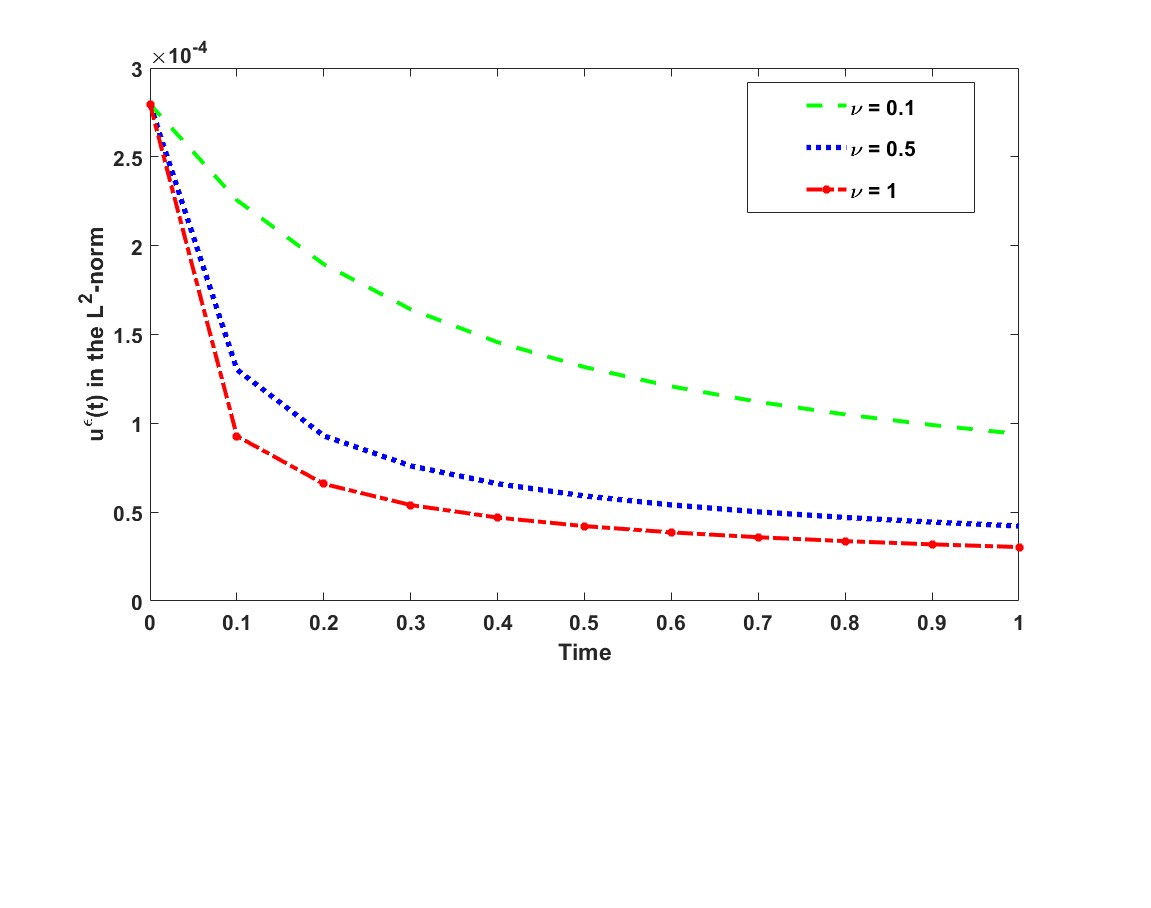}
		(iii) \includegraphics[width= 0.42\textwidth]{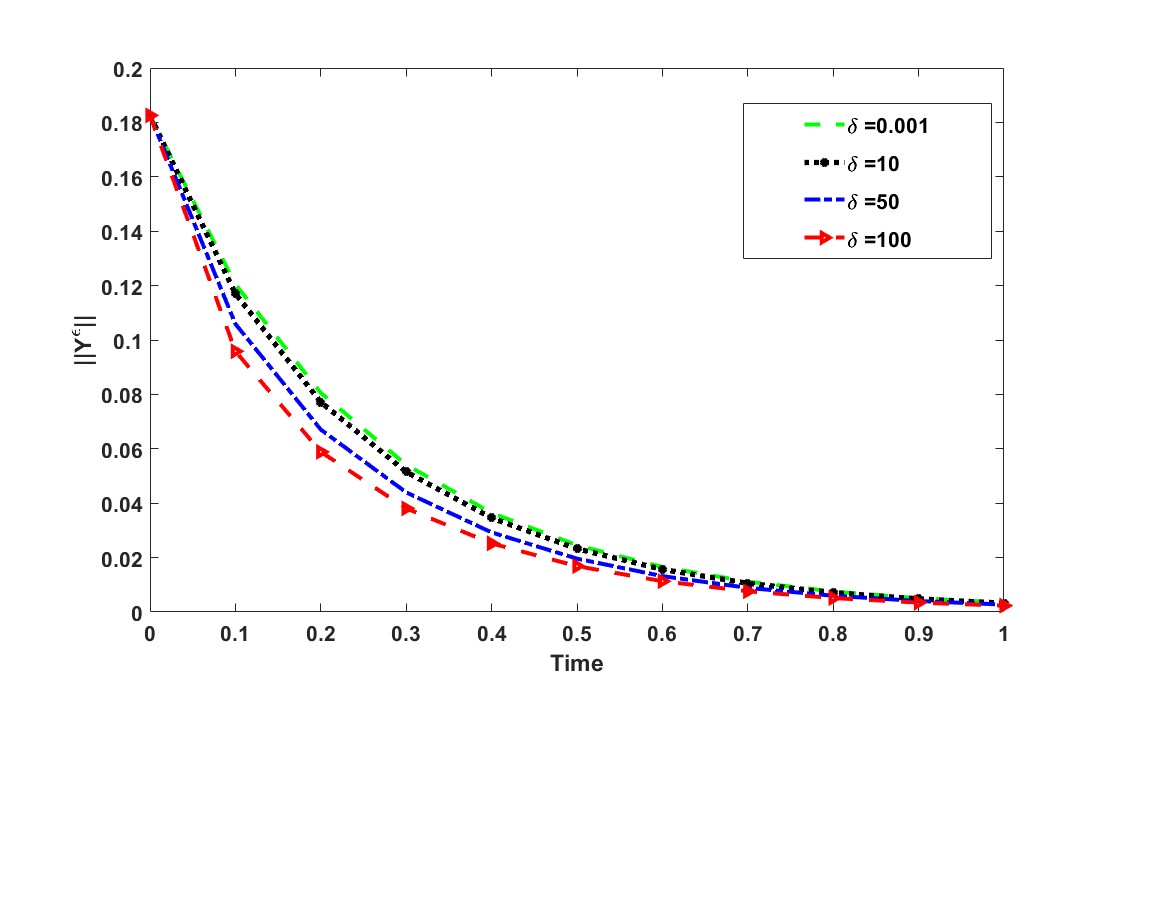}	
		(iv) \includegraphics[width= 0.42\textwidth]{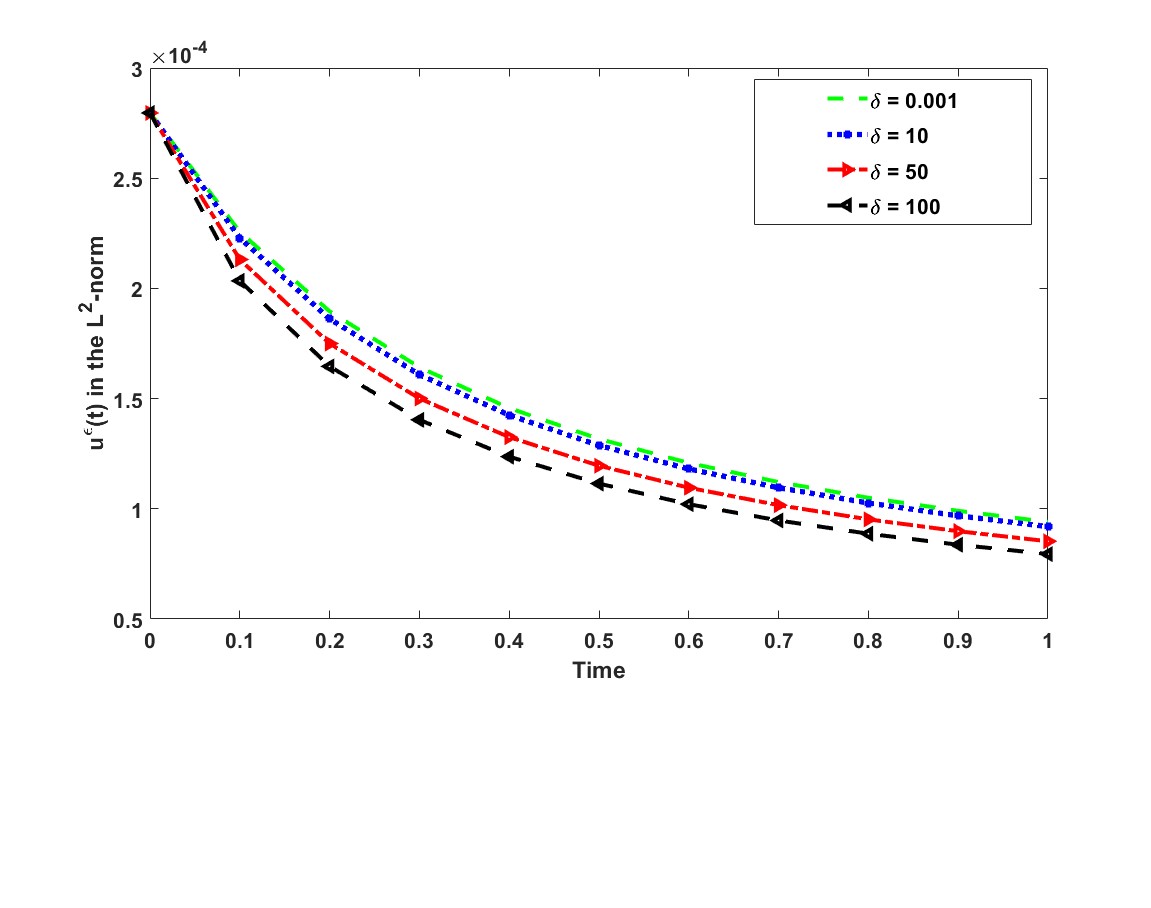}
		
		\caption{Example \ref{ex2}:
			{\bf (i)} The state variable in the $ L^{2}-$norm for different values of $ \nu $.  {\bf (ii)} Control input  in the $ L^{2}-$norm for different values of $\nu $. {\bf (iii)} The state variable in the $ L^{2}-$norm with various values of $ \delta$ with \(\nu=0.1\).  {\bf (iv)}  Control input in the $ L^{2}-$norm for various values of $\delta$. } \label{fig:ex2}
	\end{figure}    
	%\textbf{Concluding remarks.} We have proposed a penalization technique for the  Dirichlet boundary feedback control problem and established the stabilization of the resulting penalized control problem. Furthermore, we have shown that the solution of the penalized problem converges to that of the original feedback Dirichlet boundary control problem as the penalty parameter $ \epsilon $ tends to zero. A $C^{0}$-conforming finite element method applied to the penalized problem, demonstrates stabilization of the semi-discrete scheme, where the time variable remains continuous. Additionally, we have provided an error analysis for the semi-discrete scheme for both the state variable and the control input, and some numerical experiments are conducted. We observe that the best order of convergence is for \(\epsilon=ch^{l}\), when \(l=2\) from Remark \eqref{r5.2}. 
	\section*{Acknowledgments}
	Sudeep Kundu gratefully acknowledges the support of the Science \& Engineering Research Board (SERB), Government of India, under the Start-up Research Grant, Project No. SRG/2022/000360.
	\section*{Declarations}
	\textbf{Data Availability.} 
	
	The codes are available from authors on reasonable request. 
	
	\textbf{CONFlLICT OF INTEREST.} 
	
	The authors declare no conflict of interest.
    
    \textbf{Appendix.}
    \begin{proof} Proof of the existence part of Theorem \ref{th3.1}.
    
     We use a Faedo-Galerkin approximation (see \cite{chin2021study} and \cite[Chapter 3]{MR953967}) with Assumption \((A1)\). Set $A =-\frac{d^{2}}{dx^{2}}$ with dense domain 
        \[D(A)=\{\phi\in H^{2}(0,1): \phi(0)=0, \  \epsilon\phi_{x}(1)+\phi(1)=0\}
        \]  in $L^{2}(0,1).$ Further, the operator $A$ is linear selfadjoint with symmetric compact invertible operator on $L^{2}(0,1)$. Therefore, $A$ has a discrete spectrum $0<\lambda_{1}<\lambda_{2}<\ldots<\lambda_{n}<\ldots$ and $\lim_{n \to \infty}\lambda_{n}=\infty.$ The corresponding eigenfunctions \((\phi_{i})_{i\in \mathbb{N}}\) form an orthonormal basis of \(L^{2}(0,1)\).
         Assume that \((\phi_{i})_{i\in \mathbb{N}}\) be an orthogonal basis of \(H_{\{0\}}^{1}(0,1)\). We seek a function \(y_{m}^{\epsilon} :=y^{m}:[0,T]\rightarrow H_{\{0\}}^{1}(0,1) \) of the form \(y^{m}=\sum_{i=1}^{M}\alpha_{i}^{m}(t)\phi_{i}(x)\), where \(M\) is a positive integer and the coefficients  \(\alpha_{i}^{m}(t)\) are selected such that \(\alpha_{i}^{m}(0)=(y_{0}(x), \phi_{i}), \ i=1,2,\ldots, M \) and 
		\begin{align}
			\label{3.71}
			&(y_{t}^{m}, \phi)+\nu (y_{x}^{m}, \phi_{x}) + \frac{\nu}{\epsilon}y^{m}(t,1) \phi(1)+ \delta((y^{m})^{3}, \phi)= \frac{\nu}{\epsilon}u^{m}(t)\phi(1)+\alpha(y^{m}, \phi), \ \ \forall \ \phi\in V_{m} ,
		\end{align}
		where \(V_{m}= span\{\phi_{1}, \phi_{2}, \ldots, \phi_{M}\}\).
		Since \eqref{3.71} is a finite dimensional nonlinear ODE system, the existence and uniqueness directly follow from ODE theory, we obtain a local solution \(y^{m}(t)\) for \(t\in[0, t_{m}^{*})\) with \(t_{m}^{*}<T.\) 
		
		Next, for global existence, we will show  {\it{a priori}} estimates. Following the proof of Lemmas \ref{L2.1}-\ref{L2.26} with \(\gamma=0\), we can easily verify that the sequences \((y^{m}), (y_{x}^{m}), (y_{t}^{m}), (y_{xt}^{m})\), and \((y_{xx}^{m})\), which are uniformly bounded in \(L^{\infty}(0,T; L^{2}(0,1))\). Additionally, the sequence \((y^{m}(t))\) and \((y^{m}(t,1))\) are uniformly bounded in \(L^{\infty}(0,T; L^{4}(0,1))\), and \(L^{\infty}(0,T)\), respectively. Moreover, the sequence \((y_{t}^{m}(t,1))\) is uniformly bounded in \(L^{\infty}(0,T)\). 
		
		As the {\it{a priori}} bounds are uniformly bounded and independent of \(m\), using the Banach–Alaoglu theorem we can extract a subsequence considering as \((y^{m})\) such that
		\begin{align}
			\label{3.81}
			\begin{cases}
				y^{m}\rightarrow y \quad \text{weak\(^*\) in} \quad  L^{\infty}(0, T; H^{2}(0,1) ), \quad 
				y^{m}\rightarrow y \quad \text{weakly in} \quad  L^{2}(0, T; H^{2}(0,1) ),\\
				y_{t}^{m}\rightarrow y_{t} \quad \text{weak\(^*\) in} \quad L^{\infty}(0, T; H^{1}(0,1) ), \quad
				y_{t}^{m}\rightarrow y_{t} \quad \text{weakly in} \quad L^{2}(0, T; H^{2}(0,1) ), \\	
				y^{m}\rightarrow y \quad \text{weak\(^*\) in} \quad  L^{\infty}(0, T; L^{4}(0,1)), \quad
				y^{m}\rightarrow y \quad \text{weakly in} \quad  L^{4}(0, T; L^{4}(0,1)),\\
				y^{m}(t,1)\rightarrow y(t,1)\quad \text{weak\(^*\) in} \quad L^{\infty}(0,T),\quad y_{t}^{m}(t,1)\rightarrow y_{t}(t,1)\quad \text{weak\(^*\) in} \quad L^{\infty}(0,T),
			\end{cases}
		\end{align}
		as \(m\rightarrow \infty.\) Since \(H_{\{0\}}^{1}(0,1)\subset L^{2}(0,1)\) is a compact embedding, an application of the Aubin-Lions compactness lemma gives the following strong convergence:
		\begin{align}
			\label{3.91}
			y^{m}\rightarrow y \  \text{strongly in}\  L^{2}(0,T; L^{2}(0,1)) . 
		\end{align}
		
		Now fix a positive integer \(N\) and choose a test function \(\chi\in C^{1}(0,T; H_{\{0\}}^{1}(0,1))\) such that
		\begin{align}
			\label{3.10}
			\chi(t)=\sum_{i=1}^{N}\alpha_{i}(t)\phi_{i},
		\end{align}
		where \((\alpha_{i}(t))_{i=1,2,\ldots, N} \) are given smooth functions. We choose $m\geq N$. By multiplying \eqref{3.71} by \(\alpha_{i}(t)\), summing from \(i=1,2,\ldots, N\), and integrating with respect to time over \([0,T]\), we arrive at
		\begin{align}
			\label{3.11}
			\int_{0}^{T}\Big((y_{t}^{m}, \chi)+\nu (y_{x}^{m}, \chi_{x}) + \frac{\nu}{\epsilon}y^{m}(t,1) \chi(1)+ \delta((y^{m})^{3}, \chi)\Big)dt= \int_{0}^{T}\Big(\frac{\nu}{\epsilon}u^{m}(t)\chi(1)+\alpha(y^{m}, \chi)\Big)dt.
		\end{align}
		Passing the limit as \(m\rightarrow \infty\), we obtain from \eqref{3.81} and \eqref{3.11}
		\begin{align}
			\label{3.12}
			\int_{0}^{T}\Big((y_{t}, \chi)+\nu (y_{x}, \chi_{x}) + \frac{\nu}{\epsilon}y(t,1) \chi(1)+ \delta((y)^{3}, \chi)\Big)dt= \int_{0}^{T}\Big(\frac{\nu}{\epsilon}u(t)\chi(1)+\alpha(y, \chi)\Big)dt, 
		\end{align}
		provided we have \(\int_{0}^{1}(y^{m})^{3} \chi dx\rightarrow \int_{0}^{1} y^{3} \chi dx\) as \(m\rightarrow \infty\).
		
		Since \((y^{m})^{3}-y^{3}=(y^{m}-y)\big((y^{m})^{2}+y^{2}+y^{m}y\big)=\frac{1}{2}(y^{m}-y)\big((y^{m})^{2}+y^{2}+(y^{m}+y)^{2}\big)\), we can write using H\"older's inequality
		\begin{align*}
			\int_{0}^{1}\big((y^{m})^{3}-y^{3}\big)\chi dx&=\frac{1}{2} \int_{0}^{1}(y^{m}-y)\big((y^{m})^{2}+y^{2}+(y^{m}+y)^{2}\big) \chi dx,
			\\&\leq \frac{1}{2} \int_{0}^{1} (y^{m}-y)\big(3(y^{m})^{2}+3 y^{2}\big) \chi dx,
			\\&\leq C\norm{y^{m}-y}\big(\norm{y^{m}}_{L^{4}}^{2}+ \norm{y}_{L^{4}}^{2}\big) \norm{\chi}_{\infty},
			\\&\leq C\norm{y^{m}-y}\big(\norm{y^{m}}_{1}^{2}+ \norm{y}_{1}^{2}\big) \norm{\chi}_{1}\rightarrow 0 \  \text{as} \ m\rightarrow \infty.
		\end{align*}
		
		Therefore, \eqref{3.12} holds \(\forall \chi\in L^{2}(0,T; H_{\{0\}}^{1}(0,1))\). Since \(C_{0}^{\infty}(0,1)\) is dense in \(H_{\{0\}}^{1}(0,1)\), we arrive at
		\begin{align*}
			(y_{t}, \chi)+\nu (y_{x}, \chi_{x}) + \frac{\nu}{\epsilon}y(t,1) \chi(1)+ \delta((y)^{3}, \chi)= \frac{\nu}{\epsilon}u(t)\chi(1)+\alpha(y, \chi), \ \forall \ \chi\in H_{\{0\}}^{1}(0,1).
		\end{align*}
		The proof of \(y^{\epsilon}(0,x)=y_{0}(x)\) follows from \cite{chin2021study}. 
        This completes the proof.
	\end{proof}

   \begin{proof}
   Details proof of Theorem \ref{th2.2}
   
		Differentiating \eqref{ch1} with respect to time, we obtain
		\begin{align}
			\label{1.7}
			y_{tt}-\nu y_{xxt}=\alpha y_{t}-3\delta y^{2}y_{t}.
		\end{align}
		Forming the \(L^{2}\)-inner product of \eqref{1.7} with \(y_{t}\) yields
		\begin{align*}
			\frac{1}{2}\frac{d}{dt}\norm{y_{t}}^{2}+\nu \norm{y_{xt}}^{2}+ 3\delta \norm{yy_{t}}^{2}=\nu y_{xt}(t,1)y_{t}(t,1)+\alpha\norm{y_{t}}^{2}.
		\end{align*}
		Hence, we can rewrite the above as
		\begin{align*}
			\frac{1}{2}\frac{d}{dt}\norm{y_{t}}^{2}+\nu \norm{y_{xt}}^{2}+ 3\delta \norm{yy_{t}}^{2}=\nu y_{xt}(t,1)y_{t}(t,1)+\alpha\norm{y_{t}}^{2}+r(x,y_{t}(t)) (x,y_{tt}(t))-r(x,y_{t}) (x,y_{tt}(t)).
		\end{align*}
		Substituting the expression of \(y_{tt}\) from \eqref{1.7} and applying integration by parts, we obtain
		\begin{align*}
			\frac{1}{2}\frac{d}{dt}\left(\norm{y_{t}}^{2}+r(x,y_{t}(t))^{2}\right)&+\nu \norm{y_{xt}}^{2}+ 3\delta \norm{yy_{t}}^{2}\\&=\nu y_{xt}(t,1)y_{t}(t,1)+\alpha\norm{y_{t}}^{2}+r\nu (x,y_{t}(t)) y_{xt}(t,1)\\&\quad +r\nu(x,y_{t}(t)) y_{t}(t,1)+r\alpha(x,y_{t}(t))^{2}-3\delta r(x,y_{t}(t)) (x,y^{2}(t)y_{t}(t)).
		\end{align*}
		% Since \(y_{t}(t,1)=-r(x,y_{t}(t))\), it follows that
		% \begin{align*}
		% 	\frac{1}{2}\frac{d}{dt}\left(\norm{y_{t}}^{2}+r(x,y_{t}(t))^{2}\right)+\nu \norm{y_{xt}}^{2}+ 3\delta \norm{yy_{t}}^{2}\leq \alpha\norm{y_{t}}^{2}+r\alpha(x,y_{t}(t))^{2}-3\delta r(x,y_{t}(t)) (x, y^{2}(t)y_{t}(t)).
		% \end{align*}
		% An application of the Young's inequality gives
       
		% \begin{align*}
		% 	\frac{1}{2}\frac{d}{dt}\left(\norm{y_{t}}^{2}+r(x,y_{t}(t))^{2}\right)+\nu \norm{y_{xt}}^{2}+ 3\delta \norm{yy_{t}}^{2}\leq \alpha\norm{y_{t}}^{2}+ C(x,y_{t}(t))^{2}+C(r, \delta)\norm{y}_{\infty}^{4}\norm{y_{t}}^{2}.
		% \end{align*}
		% Multiplying to the above inequality by \(2e^{2\gamma t}\), we have
         Using \(y_{t}(t,1)=-r(x,y_{t}(t))\), Young's inequality, and multiplying the above inequality by \(2e^{2\gamma t}\), we have 
		\begin{align*}
			\frac{d}{dt}\Big(e^{2\gamma t}\norm{y_{t}}^{2}&+re^{2\gamma t}(x,y_{t}(t))^{2}\Big)+2\nu e^{2\gamma t} \norm{y_{xt}}^{2}+ 6\delta e^{2\gamma t} \norm{yy_{t}}^{2}\\&\leq 2\alpha e^{2\gamma t}\norm{y_{t}}^{2}+ Ce^{2\gamma t}(x,y_{t}(t))^{2}+C(r, \delta)e^{2\gamma t}\norm{y}_{\infty}^{4}\norm{y_{t}}^{2}+ 2 \gamma e^{2\gamma t}\norm{y_{t}}^{2}.
		\end{align*}
		Integrating to the above inequality with respect to time from \(0\) to \(t\), and using Lemma \ref{L1.1} together with the estimate \(\norm{y}_{\infty}\leq C\norm{y}_{1}\).
        Also, using Sobolev embedding, we can write from \eqref{ch1}
        \begin{align*}
        \norm{y_{t}}^{2}\leq C\big(\norm{y_{xx}}^{2}+\norm{y}^{2}+\norm{y}_{L^{6}}^{6}\big)\leq C\big(\norm{y_{xx}}^{2}+\norm{y}^{2}+\norm{y}_{1}^{6}\big).
        \end{align*}
        Setting \(t=0\) to the above inequality and multiplying by \(e^{-2\gamma t}\) in the resulting inequality completes the first part of the proof.
    
		For the proof of second part, taking \(L^{2}\)-inner product between \eqref{1.7} and \(y_{tt}\), we obtain
		\begin{align*}
			\frac{\nu}{2}\frac{d}{dt}\norm{y_{xt}}^{2}+\norm{y_{tt}}^{2}=\nu y_{xt}(t,1)y_{tt}(t,1)+\alpha (y_{t}, y_{tt})-3\delta(y^{2}y_{t}, y_{tt}).
		\end{align*}
		Adding \(r(x, y_{tt}(t))^{2}\) to both sides of the above equation, we arrive at
		\begin{align*}
			\frac{\nu}{2}\frac{d}{dt}\norm{y_{xt}}^{2}+\norm{y_{tt}}^{2}+r(x,y_{tt}(t))^{2}=\nu y_{xt}(t,1)y_{tt}(t,1)+\alpha (y_{t}, y_{tt})-3\delta(y^{2}y_{t}, y_{tt})+r(x,y_{tt}(t)) (x,y_{tt}(t)).
		\end{align*}
		Now, substituting the value of \(y_{tt}\) from \eqref{1.7} and using integration by parts, we get
		\begin{align*}
			\frac{\nu}{2}\frac{d}{dt}\norm{y_{xt}}^{2}+\norm{y_{tt}}^{2}&+r\big(x,y_{tt}(t)\big)^{2}\\&=\nu y_{xt}(t,1)y_{tt}(t,1)+\alpha (y_{t}, y_{tt})-3\delta(y^{2}y_{t}, y_{tt})+r\nu(x,y_{tt}(t)) y_{xt}(t, 1)\\&\quad -r\nu (x,y_{tt}(t)) y_{t}(t, 1)+r\alpha(x,y_{tt}(t)) (x,y_{t}(t))-3r\delta (x,y_{tt}(t)) (x, y^{2}(t)y_{t}(t)).
		\end{align*}
		A use of the Young's inequality yields
		\begin{align*}
			\frac{\nu}{2}\frac{d}{dt}\norm{y_{xt}}^{2}+\frac{1}{2}\norm{y_{tt}}^{2}+\frac{r}{2}\big(x,y_{tt}(t)\big)^{2}\leq \alpha^{2}\norm{y_{t}}^{2}+C(\delta,r)\norm{y}_{\infty}^{4}\norm{y_{t}}^{2}+C(r,\alpha,\nu)(x,y_{t}(t))^{2}.
		\end{align*}
		Multiplying the above inequality by \(2e^{2\gamma t}\) and integrating the resulting inequality with respect to time from \(0\) to \(t\), we observe that
		\begin{align}
			\label{1.8}
			\nu e^{2\gamma t}\norm{y_{xt}}^{2}+\int_{0}^{t}e^{2\gamma s}\Big(\norm{y_{tt}(s)}^{2}&+r(x,y_{tt}(s))^{2}\Big)ds\\& \nonumber \leq C\int_{0}^{t}e^{2\gamma s}\Big(\norm{y_{t}(s)}^{2}+(x,y_{t}(s))^{2}\Big)ds+\norm{y_{xt}(0}^{2}.
		\end{align}
		Differentiating \eqref{ch1} with respect to \(x\) and using Sobolev embedding, we can write
        \begin{align*}
            \norm{y_{xt}}^{2}\leq C\big(\norm{y_{xxx}}^{2}+\norm{y}_{1}^{2}+\norm{y}_{1}^{6} \big),
        \end{align*}
        and after evaluating at \(t=0\) yields
		\begin{align}
			\label{1.9}
			\norm{y_{xt}(0}^{2}\leq C(\norm{y_{0}}_{3}^{2}+\norm{y_{0}}_{1}^{6}).
		\end{align}
		Consequently, the result follows from \eqref{1.8} using \eqref{1.9} and Lemma \ref{L1.1} with the first part, and finally multiplying the resulting inequality by \(e^{-2\gamma t}\).
	\end{proof}
	

\begin{thebibliography}{9}	

		\bibitem{azouani2013feedback} 
		{A. Azouani, and E. S. Titi.}
		{ Feedback control of nonlinear dissipative systems by finite determining parameters - A reaction-diffusion paradigm}, {Evolution Equations and Control Theory}, 3, 579-594, 2014.
		
		\bibitem{MR1874072}
		{ N. Arada, and  J. P. Raymond}.
		{Dirichlet boundary control of semilinear parabolic equations parts
			{1}: {P}roblems with no state constraints},
		{Applied Mathematics and Optimization},
		{45}, {125--143},
		{2002}.

		
		\bibitem{MR351118}
		{I. Babu\v ska}.
		{The finite element method with penalty},
		{Mathematics of Computation},
		{27}, {221--228},
		{1973}.
		
		\bibitem{MR853660}
		{J. W. Barrett, and C. M. Elliott}.
		{Finite element approximation of the {D}irichlet problem using
			the boundary penalty method},
		{Numerische Mathematik},
		{49}, {343--366},
		{1986}.
		\bibitem{belgacem2003singular}
		{ F. B. Belgacem,  H. E. Fekih, and  H. Metoui}.
		{Singular perturbation for the Dirichlet boundary control of elliptic problems},
		{ESAIM: Mathematical Modelling and Numerical Analysis},
		{37},
		{833--850},
		{2003}.
		
		\bibitem{Martin2023}
		{N. Berglund}. {An Introduction to Singular Stochastic PDEs: Allen–Cahn Equations, Metastability, and Regularity Structures}, European Mathematical Society, 
		129, 56–57, 2023.
		
		
		\bibitem{ben2003penalized}
		{F. B. Belgacem, H. E. Fekih, and J. P. Raymond}.
		{A penalized Robin approach for solving a parabolic equation with nonsmooth Dirichlet boundary conditions},
		{Asymptotic Analysis},
		{34},
		{121--136},
		{2003}.	
		
		\bibitem{ben2011dirichlet}
		{F. B. Belgacem,  C. Bernardi, and H. E. Fekih}.
		{Dirichlet boundary control for a parabolic equation with a final observation I: A space--time mixed formulation and penalization},
		{Asymptotic Analysis},
		{71},
		{101--121},
		{2011}.		 
		
		
		\bibitem{MR4399840}
		{B. Bir, D. Goswami, and A. K. Pani}.
		{Finite element penalty method for the {O}ldroyd model of order
			one with non-smooth initial data},
		{Computational Methods in Applied Mathematics},
		{22}, {297--325},
		{2022}.
		
	\bibitem{CHAFEE1975111}
	{N. Chafee.}
	{Asymptotic behavior for solutions of a one-dimensional parabolic equation with homogeneous Neumann boundary conditions},
	{Journal of Differential Equations},
	{18},
	{111-134},
	{1975}.	

  \bibitem{MR3590173}
     {J.P. Chehab,  A. A. Franco, and Y. Mammeri}.
     {Boundary control of the number of interfaces for the
              one-dimensional {A}llen-{C}ahn equation},
  {Discrete and Continuous Dynamical Systems. Series S},
    {10}, {87--100},
      {2017}.

    
		\bibitem{casas2009penalization}
		{E. Casas,   M. Mateos, and J. P. Raymond}.
		{Penalization of Dirichlet optimal control problems},
		{ESAIM: Control, Optimisation and Calculus of Variations},
		{15},
		{782--809},
		{2009}.
		
		
		\bibitem{CAREY1984183}
		{G. F. Carey and R. Krishnan}.
		{Penalty finite element method for the Navier-Stokes equations},
		{Computer Methods in Applied Mechanics and Engineering},
		{42}, {183-224},
		{1984}.
		
		\bibitem{chin2021study}
		{ P. W. M. Chin.}
		{The study of the numerical treatment of the real Ginzburg--Landau equation using the Galerkin method},
		{Numerical Functional Analysis and Optimization}, {42}, {1154--1177}, {2021}. 	 
		
		
		\bibitem{JDeng2019}
		{J. Deng and  Z. Si}.
		{A decoupling penalty finite element method for the stationary incompressible Magnetohydrodynamics equation},
		{International Journal of Heat and Mass Transfer},
		{ 128},
		{601-612},
		{2019}.
		
		\bibitem{dione2015penalty}
		{I. Dione and J. M Urquiza}.
		{Penalty: finite element approximation of Stokes equations with slip boundary conditions},
		{Numerische Mathematik},
		{129},
		{587--610},
		{2015}.
		
		
		\bibitem{MR1625845}
		{L. C. Evans. }
		{Partial differential equations},
		{Graduate Studies in Mathematics},
		{American Mathematical Society, Providence, RI}, {19}, {xviii+662}, {1998}.
		
		\bibitem{9304185}
		{P. Guzmán, and C. Prieur}.
		{Rapid stabilization of a reaction-diffusion equation with distributed disturbance}, 
		{2020 59th IEEE Conference on Decision and Control (CDC)}, 
		{666-671}, {2020}.
		
		\bibitem{gudi2014}
		{T. Gudi}.
		{Babu\v ska’s penalty method for inhomogeneous Dirichlet problem: error estimates and multigrid algorithms},
		{International Journal of Numerical Analysis Modeling Series B},
		{5},	{299--316},
		{2014}.
		
		\bibitem{S0363012996304870}
		{L. S. Hou, and S. S. Ravindran}.
		{A penalized Neumann control approach for solving an optimal Dirichlet control problem for the Navier--Stokes equations},
		{SIAM Journal on Control and Optimization},
		{36},
		{1795-1814},
		{1998}.
		
		\bibitem{S1064827597325153}
		{L. S. Hou, and S. S. Ravindran}.
		{Numerical approximation of optimal flow control problems by a penalty method: error estimates and numerical results},
		{SIAM Journal on Scientific Computing},
		{20},
		{1753-1777},
		{1999}.
		
	
		
		\bibitem{MR2136999}
		{Y. He}.
		{Optimal error estimate of the penalty finite element method
			for the time-dependent {N}avier-{S}tokes equations},
		{Mathematics of Computation},
		{74}, {1201--1216},
		{2005}.
		
		\bibitem{HE2010708}
		{Y. He and J. Li}.
		{A penalty finite element method based on the Euler implicit/explicit scheme for the time-dependent Navier–Stokes equations},
		{Journal of Computational and Applied Mathematics},
		{235},
		{708-725},
		{2010}.
		\bibitem{7426756}	
		{T. Hashimoto and M. Krstic}.
		{Stabilization of reaction diffusion equations with state delay using boundary control input}, {IEEE Transactions on Automatic Control},	{61},	
		{4041-4047}, 	
		{2016}.
		
		\bibitem{MR2732467}
		{S. Kondo and T. Miura}.
		{Reaction-diffusion model as a framework for understanding
			biological pattern formation},
		{American Association for the Advancement of Science. Science},
		{329}, {1616--1620},
		{2010}.
		
		
		
		\bibitem{19M1252235}
		{I. Karafyllis and M. Krstic}.
		{Global stabilization of a class of nonlinear reaction-diffusion partial differential equations by boundary feedback},
		{SIAM Journal on Control and Optimization},
		{57}, {3723-3748}, {2019}.
		\bibitem{18M1213129}
		{I. Karafyllis and M. Krstic}.
		{Small-gain-based boundary feedback design for global exponential stabilization of one-dimensional semilinear parabolic PDEs},
		{SIAM Journal on Control and Optimization},
		{57},
		{2016-2036},
		{2019}.
		
		\bibitem{MR3790146}
		S. Kundu and A. K. Pani.
		{Finite element approximation to global stabilization of the
			{B}urgers' equation by {N}eumann boundary feedback control
			law},
		Advances in Computational Mathematics,
		44, 541--570,
		2018.
		
		\bibitem{kang1991unbounded}
		{S. Kang, K. Ito, and  J. A. Burns.}
		{Unbounded observation and boundary control problems for Burgers equation},
		{[1991] Proceedings of the 30th IEEE Conference on Decision and Control},
		{2687--2692},
		{1991}.
		\bibitem{LHACHEMI2022109955}
		{H. Lhachemi and C. Prieur}.
		{Finite-dimensional observer-based boundary stabilization of reaction–diffusion equations with either a Dirichlet or Neumann boundary measurement},
		{Automatica},
		{135},
		{109955},
		{2022}.
		\bibitem{9099005}
		{H. Lhachemi, C. Prieur, and E. Trélat}.
		{PI regulation of a reaction–diffusion equation with delayed boundary control}, 
		{IEEE Transactions on Automatic Control}, 
		{66},
		{1573-1587},  {2021}.
		
		
		\bibitem{MR4077349}
		{H. Lhachemi and R. Shorten}.
		{Boundary feedback stabilization of a reaction-diffusion
			equation with {R}obin boundary conditions and state-delay},
		{Automatica. A Journal of IFAC, the International Federation of
			Automatic Control},
		{116}, {108931, 9},
		{2020}.
        
		\bibitem{MR2912615}
		{Y. Li and R. An}.
		{Penalty finite element method for {N}avier-{S}tokes equations
			with nonlinear slip boundary conditions},
		{International Journal for Numerical Methods in Fluids},
		{69}, {550--566},
		{2012}.
		
		\bibitem{MR1225705}
		A. K. Pani.
		{A finite element method for a diffusion equation with
			constrained energy and nonlinear boundary conditions},
		{Australian Mathematical Society. Journal. Series B. Applied
			Mathematics},
		{35}, {87--102},
		{1993}.
		\bibitem{8392394}
		{C. Prieur and E. Trélat}.		 
		{Feedback stabilization of a 1-D linear reaction–diffusion equation with delay boundary control}, {IEEE Transactions on Automatic Control},
		{64},
		{1415-1425}, {2019}.
		
		
		\bibitem{MR4020518}
		{J. Qi, M. Krstic, and S. Wang}.
		{Stabilization of reaction-diffusions {PDE} with delayed
			distributed actuation},
		{Systems \& Control Letters},
		{133},  {104558, 8},
		{2019}.
		
		
		\bibitem{MR660571}
		{ J. N. Reddy}.
		{On penalty function methods in the finite-element analysis of
			flow problems},
		{International Journal for Numerical Methods in Fluids},
		{2}, {151--171},
		{1982}.
		
		\bibitem{ravindran2017finite}
		{S. S. Ravindran}.
		{Finite element approximation of Dirichlet control using boundary penalty method for unsteady Navier--Stokes equations},
		{ESAIM: Mathematical Modelling and Numerical Analysis},
		{51},
		{825--849},
		{2017}.
		
		\bibitem{18M1233716}
		{S. S. Ravindran}.
		{Penalization of dirichlet boundary control for nonstationary Magneto-hydrodynamics},
		{SIAM Journal on Control and Optimization},
		{58},
		{2354-2382},
		{2020}.	
		
		\bibitem{0732016}
		{J. Shen}.
		{On error estimates of the penalty method for unsteady Navier–Stokes equations},
		{SIAM Journal on Numerical Analysis},
		{32}, {386-403},
		{1995}.
		
		\bibitem{MR3486328}
		{H. Su, X. Feng, and P. Huang}.
		{Iterative methods in penalty finite element discretization for
			the steady {MHD} equations},
		{Computer Methods in Applied Mechanics and Engineering}, 
		{304}, {521--545},
		{2016}.
		\bibitem{MR3969002}
		{H. Su, S. Mao, and X. Feng}.
		{Optimal error estimates of penalty based iterative methods for
			steady incompressible magnetohydrodynamics equations with
			different viscosities},
		{Journal of Scientific Computing},
		{79}, {1078--1110},
		{2019}.
	
	\bibitem{SADAF2023107097}
	{M. Sadaf, S. Arshed,  G. Akram,  M. R. Ali, and I. Bano}.
	{Analytical investigation and graphical simulations for the solitary wave behavior of Chaffee–Infante equation},
	{Results in Physics},
	{54},
	{107097},
	{2023}.
		
		
		\bibitem{MR953967}
		{R. Temam}.
		{Infinite-dimensional dynamical systems in mechanics and
			physics},
		{Applied Mathematical Sciences},
		{68},  {xxii+648},
		{Springer-Verlag, New York}, 
		{1997}.

		
		\bibitem{MR0534334}
		{V. Thomee}.
		{Galerkin-finite element method for parabolic equations},
		{Springer Series in Computational Mathematics},
		{25},
		{xii+370},
		{Springer-Verlag, Berlin},
		{2006}.
		
		\bibitem{MR4721995}
		{G. Veeraragavan and S. Radhakrishnan}.
		{Boundary controller design for stabilization of stochastic
			nonlinear reaction-diffusion systems with time-varying delays},
		{Computational Methods for Differential Equations},
		{12},  {196--206},
		{2024}.
		
		
		\bibitem{MR2754643}
		{K. Wang, Y. Shang,  and R. Zhao}.
		{Optimal error estimates of the penalty method for the
			linearized viscoelastic flows},
		{International Journal of Computer Mathematics},
		{87}, {3236--3253},
		{2010}.
		\bibitem{MR2818046}
		{K. Wang, Y. He, and X. Feng}.
		{On error estimates of the fully discrete penalty method for
			the viscoelastic flow problem},
		{International Journal of Computer Mathematics},
		{88}, {2199--2220},
		{2011}.	
		
		\bibitem{MR3762670}
		{K. N. Wu, H. X. Sun, P. Shi, and C. C. Lim}.
		{Finite-time boundary stabilization of reaction-diffusion
			systems},
		{International Journal of Robust and Nonlinear Control},
		{28},  {1641--1652},
		{2018}.
		
		
		
		\bibitem{yang2019asymptotic}
		{K. Yang, J. Zheng, and G. Zhu.}
		{Asymptotic output tracking for a class of semilinear parabolic equations: A semianalytical approach},
		{International Journal of Robust and Nonlinear Control}, {29}, {2471--2493}, {2019}. 
		
		
	\end{thebibliography}
\end{document}